\documentclass[11pt]{article}

\usepackage{geometry}
\usepackage{authblk}
\usepackage{lmodern}
\usepackage{amsmath}\usepackage{amscd}
\usepackage{amssymb}
\usepackage{theorem}
\usepackage{xypic}
\usepackage[all,v2]{xy}
\xyoption{2cell}
\UseAllTwocells
\usepackage{enumitem}
\usepackage{rotating}
\usepackage{mathrsfs}
\usepackage{stmaryrd}
\usepackage{mathdots}
\usepackage{hyperref}
\usepackage{caption}
\usepackage[normalem]{ulem}
\usepackage{xcolor}

\usepackage{tikz-cd}

\usepackage{subcaption}
\usepackage{tikz}

\usepackage{fixltx2e}

\newenvironment{items}
{\begin{enumerate}[topsep=3pt, itemsep=3pt, parsep=0pt, label=(\roman*)]}
{\end{enumerate}}

\renewcommand{\tilde}{\widetilde}
\newcommand{\dual}[1]{#1^{\vee}}

\newcommand{\rel}{{\mbox{\tiny rel}}}
\newcommand{\lin}{{\mbox{\tiny lin}}}
\newcommand{\con}{{\mbox{\tiny con}}}

\newcommand{\argument}{{{\,\cdot\,}}}

\newcommand{\XX}{{\mathfrak X}}
\renewcommand{\AA}{{\mathfrak A}}

\newcommand{\A}{{\mathscr A}}
\newcommand{\B}{{\mathscr B}}

\newcommand{\C}{{\mathscr C}}

\newcommand{\M}{{\mathscr M}}

\newcommand{\N}{{\mathscr N}}

\renewcommand{\P}{{\mathscr P}}

\newcommand{\LL}{{\mathfrak L}}

\newcommand{\zz}{{\mathbb Z}}

\newcommand{\rr}{{\mathbb R}}

\renewcommand{\O}{{\mathscr O}}

\newcommand{\ev}{\mathop{\rm ev}\nolimits}

\newcommand{\Mor}{\mathop{\rm Mor}\nolimits}

\newcommand{\Map}{\mathop{{\rm Map}}\nolimits}

\newcommand{\iso}{\stackrel{\sim}{\rightarrow}}

\newcommand{\doublearrowstack}[2]%
                      {{{{\scriptstyle#1}\atop{\textstyle\longrightarrow}}\atop{{\textstyle\longrightarrow}\atop{\scriptstyle#2}}}}
\newcommand{\rightleftarrowstack}[2]%
                      {{{{\scriptstyle#1}\atop{\textstyle\longrightarrow}}\atop{{\textstyle\longleftarrow}\atop{\scriptstyle#2}}}}
\newcommand{\leftrightarrowstack}[2]%
                      {{{{\scriptstyle#1}\atop{\textstyle\longleftarrow}}\atop{{\textstyle\longrightarrow}\atop{\scriptstyle#2}}}}

\newtheorem{thm}{Theorem}[section]
\newtheorem{cor}[thm]{Corollary}
\newtheorem{lem}[thm]{Lemma}
\newtheorem{prop}[thm]{Proposition}

\theorembodyfont{\upshape}
\newtheorem{defn}[thm]{Definition}

\newtheorem{rmk}[thm]{Remark}

\newtheorem{ex}[thm]{Example}

\newenvironment{pf}{\begin{trivlist}\item[]{\sc Proof.}}%
            {\nolinebreak $\Box$ \end{trivlist}}

\newenvironment{pfLooPathSp}{\begin{trivlist}\item[]{\sc Proof of Proposition~\ref{pro:paris}.}}%
            {\nolinebreak $\Box$ \end{trivlist}}

            {\nolinebreak $\Box$ \end{trivlist}}

\newenvironment{proof}{\begin{trivlist}\item[]{\sc Proof.}}%
            {\nolinebreak $\Box$ \end{trivlist}}

\DeclareMathOperator\id{id}

\newcommand{\im}{\mathop{\rm im}\nolimits}
\newcommand{\rk}{\mathop{\rm rk}\nolimits}
\newcommand{\pr}{\mathop{\rm pr}\nolimits}

\newcommand{\Sym}{\mathop{\rm Sym}\nolimits}

\newcommand{\noprint}[1]{}

\newcommand{\xxto}[1]{\xrightarrow{#1}}
\newcommand{\mapnm}{{\Map(\N, \M)}}
\newcommand{\Z}{{\mathscr Z}}
\newcommand{\pathast}{\Gamma(I,a^\ast(TM\oplus L)\,dt\oplus a^\ast L)} 
\newcommand{\eetale}{a local diffeomorphism}

\newcommand{\fiber}{\text{fiber}}

\newcommand{\Ho}{\mathrm{Ho}}
\newcommand{\ddim}{\dim^{\text vir}}
\newcommand{\cntnL}{\nabla}
\newcommand{\cntnM}{{\nabla^{ M}}}

\renewcommand{\xto}[1]{\xrightarrow{#1}}
\newcommand{\onto}{\twoheadrightarrow}

\newcommand{\baseDI}{\breve{B}}
\newcommand{\totalDI}{\breve{P}}
\newcommand{\strDI}{\breve{D}}

\title{Differential graded manifolds of finite positive amplitude}

\author{Kai Behrend,
Hsuan-Yi Liao
and Ping Xu}

\date{}

\newcommand{\Addresses}{{
  \bigskip
  \footnotesize

  Kai Behrend, \textsc{Department of Mathematics,  University of British Columbia}\par\nopagebreak
  \textit{E-mail address}: \texttt{behrend@math.ubc.ca}

  \medskip

  Hsuan-Yi Liao, \textsc{Department of Mathematics, 
National Tsing Hua University}\par\nopagebreak
  \textit{E-mail address}: \texttt{hyliao@math.nthu.edu.tw}

  \medskip

  Ping Xu, \textsc{Department of Mathematics, Pennsylvania State University}\par\nopagebreak
  \textit{E-mail address}: \texttt{ping@math.psu.edu}

}}

\begin{document}
\sloppy

\maketitle


\begin{abstract}
{We prove that dg manifolds of finite positive amplitude, i.e.\ bundles of  positively graded curved
 $L_\infty[1]$-algebras, form a category of fibrant objects. 
As a main step in the proof, we obtain a factorization theorem using path spaces. 
First we construct an infinite-dimensional factorization of a diagonal morphism using actual path spaces motivated by the AKSZ construction. Then we cut down to finite dimensions using the Fiorenza-Manetti method. The main ingredient in our method is the homotopy transfer theorem for curved $L_\infty[1]$-algebras. 
As an application, we study the derived intersections of manifolds. 
} 
\end{abstract}

\let\thefootnote\relax\footnotetext{Research partially supported by NSF grants
 DMS-1707545 and DMS-2001599, KIAS Individual Grant MG072801 and MOST/NSTC Grant
110-2115-M-007-001-MY2.}

\tableofcontents

\section{Introduction}

This work is a contribution to the theory of derived manifolds in the
context of $C^\infty$-geometry. Recently,  there has been a growing interest
in derived differential geometry. The main purpose is to use  ``homotopy fibered product" in a proper sense to replace
 the fibered product, which is not always defined in classical differential
geometry. There has appeared quite many works in this 
direction in the literature  mainly motivated by derived algebraic geometry
of Lurie and To\"{e}n-Vezzosi \cite{MR2717174, MR2394633}--- see also \cite{MR3033634, eugster2021introduction}.
 For instance, see \cite{MR2641940,carchedi2012homological, MR3121621, MR3221297, borisov2011simplicial,  2018arXiv180407622P, macpherson2017universal, Nuiten, eugster2021introduction,2019arXiv190506195C}
for the $C^\infty$-setting and \cite{MR4036665, MR3959070} for
the analytic-setting.  
Our approach is based on
the geometry of dg manifolds of positive amplitude, or equivalently, bundles of positively graded curved $L_\infty[1]$-algebras. In spirit, it is analogous
to  the dg approach to derived algebraic geometry
 \cite{2002math.....12225B,2002math.....12226B,MR1801413,MR1839580,MR2496057}
and is closer to the approach of Carchedi-Roytenberg  \cite{carchedi2012homological,MR3121621}, by  avoiding the machinery of  $C^\infty$-ring.

For us, an {\em $L_\infty$-bundle} is a triple $\M=(M,L,\lambda)$,
where $M$ is a $C^\infty$-manifold, $L=L^1\oplus\ldots\oplus L^n$ is a finite-dimensional graded vector bundle over $M$, and $\lambda=(\lambda_k)_{0\leq k< n}$ is a smooth family of multilinear operations $\lambda_k:L^{\otimes k}\to L$ of degree $1$, making each fiber $L|_P$, where $P\in M$, a curved $L_\infty[1]$-algebra. 
 Equivalently, $\Sym_{\O_M}L^\vee$ is a sheaf of commutative differential graded algebras over the sheaf $\O_M$ of $C^\infty$-functions on $M$; that is, $L$ is a dg manifold of finite positive amplitude. We will mainly use the formulation of $L_\infty$-bundles since it is more geometric and more appropriate for our main construction.
 
We prove that $L_\infty$-bundles form a category of fibrant objects (see \cite{MR341469} for details on categories of fibrant objects) which contains the category of $C^\infty$-manifolds as a full subcategory.  
Therefore, we can make sense of ``homotopy fibered product'' for $L_\infty$-bundles. In particular, in the homotopy category of $L_\infty$-bundles, we can talk about derived intersection of two submanifolds in a smooth manifold $M$. See Section~\ref{sec:HptFibProd}. 

Recall that a {\em category of fibrant objects }is a category $\C$, together with two subcategories, the category of fibrations in $\C$, and the category of weak equivalences in $\C$, subject to a list of axioms (see Definition~\ref{def:CFO}). 
One is interested in the localization of $\C$ at the weak equivalences, i.e.\ the category obtained from $\C$ by formally inverting the weak equivalences.  This localization process
 gives rise to an $\infty$-category, whose associated 1-category is the {\em homotopy category }of $\C$.  The presence of the fibrations, satisfying the above axioms, simplifies the description of the localized category significantly.  In fact, Cisinsky \cite{MR2746284} has shown that the homotopy type of the space of morphisms from $X$ to $Y$ can be represented by the category of spans from $X$ to $Y$, where a span $X\leftarrow X'\to Y$ has the property that the backwards map $X'\to X$ is a trivial fibration.

The zeroth operation $\lambda_0$ is a global section of the vector bundle $L^1$ over $M$.  We call the zero locus of $\lambda_0$ the {\em classical locus }of the $L_\infty$-bundle $\M=(M,L,\lambda)$. At a classical point $P\in M$, the tangent map of $\lambda_0$ induces a map $D_P \lambda_0:TM|_P \to L^1|_P$ (see Section~\ref{sec:Etale} for details), and we can define the {\em tangent complex}
$$\xymatrix{
TM|_P\rto^{D_P\lambda_0}& L^1|_P\rto^{\lambda_1|_P}& L^2|_P\rto^{\lambda_1|_P}&\ldots}.$$
We define a morphism of $L_\infty$-bundles to be a {\em weak equivalence }if
\begin{items}
\item it induces a bijection on classical loci,
\item  the linear part  induces a quasi-isomorphism on tangent complexes at all classical points.
\end{items}
We define a morphism of $L_\infty$-bundles to be a {\em fibration}, if 
\begin{items}
\item the underlying morphism of manifolds is a differentiable submersion,
\item the linear part of the morphism of curved $L_\infty[1]$-algebras
 is a degreewise surjective morphism of graded vector bundles.
\end{items}
Our main theorem is the following:
\begin{trivlist}
\item {\bf Theorem A} (Theorem~\ref{thm:main})
 {\it The category of $L_\infty$-bundles is a category of fibrant objects.}
\end{trivlist}

Restricting, for example,  to $L_\infty$-bundles where $M$ is a single point,
 which is classical 
\cite{MR3832143,FioMan},
 we obtain the category of fibrant objects of
 finite-dimensional, positively graded, hence nilpotent $L_\infty[1]$-algebras,
 where weak equivalences are quasi-isomorphisms, and fibrations are 
degreewise surjections  (see also \cite{MR1843805, MR2628795}).
On the other extreme, restricting to $L_\infty$-bundles with $L=0$, we obtain
 the category of smooth manifolds, with diffeomorphisms for weak equivalences and differentiable submersions as fibrations.
Our choice of definition of  $L_\infty$-bundles
 is chosen to be minimal, subject to including both of these extreme cases.

Most of the proof of Theorem~A is straightforward.  The non-trivial part is the proof of the factorization property.  It is well-known, that it is sufficient to prove factorization for diagonals:

\begin{trivlist}
\item {\bf Theorem B} (Theorem~\ref{thm:Seoul})
 {\it The diagonal morphism $\M \to \M \times \M$ of an $L_\infty$-bundle can be factorized as the composition of a fibration and a weak equivalence.}
\end{trivlist}

 The proof of Theorem~B proceeds in two stages.

First, we factor the diagonal of an $L_\infty$-bundle $\M$ as 
$$\M\longrightarrow PT\M[-1]\longrightarrow\M\times\M\,.$$
Here, $PT\M[-1]$ is the path space $\Map(I,T\M[-1])$ 
of the shifted tangent bundle $T\M[-1]$, where $I$ is an open interval containing $[0,1]$.  
Another way to view $\Map(I,T\M[-1])$ is as $\Map(TI[1],\M)$, where $TI[1]$ is the shifted tangent bundle of the interval $I$. Note
that  $TI[1]$ is not an $L_\infty$-bundle in our sense (the corresponding graded vector bundle is not concentrated in positive degrees);
this is essentially the viewpoint of  AKSZ construction \cite{MR1432574}. See 
Appendix~\ref{sec:AKSZ}.

The path space $PT\M[-1]$ is infinite-dimensional, so in a second step we pass to a finite-dimensional model by using homotopical perturbation theory, in this case the curved $L_\infty[1]$-transfer theorem  \cite{ezra,MR1950958,MR3276839}.  

In fact, it is convenient to break up this second step into two subitems. 
The first step is to restrict $PT\M[-1]$ to $P_gM\subset PM$, a manifold of {\em short geodesic paths} with respect to an affine connection on $M$, where $PM=\Map(I,M)$ is the base manifold of $PT\M[-1]$.
 This means that we restrict to paths which are sufficiently short (or, in other words, sufficiently slowly traversed) so that evaluation at $0$ and $1$ defines an open embedding $P_gM\to M\times M$.  
Note that as a smooth manifold
$P_gM$ can be identified with an open neighborhood of the zero section
of $TM$, the open embedding $P_gM\to M\times M$ follows from the
tubular  neighborhood theorem of smooth manifolds.

In the second step, we choose a linear connection on $L$, and apply the $L_\infty[1]$-transfer theorem fiberwise over $P_gM$ to cut down the graded vector bundle to finite dimensions. 
Our method in this step is inspired by the Fiorenza-Manetti
approach in \cite{MR3832143, FioMan}.

After this process, we are left with the following  finite-dimensional
\emph{derived  path space $\P\M$}. The base manifold is
 $P_gM\subset M\times M$. Over this we have the
 graded vector bundle
\begin{equation}\label{grvb}
P_\con(TM\oplus L)\,dt\oplus P_\lin L\,.
\end{equation}
The fiber over a (short geodesic) path $a:I\to M$ of this graded vector bundle is the direct sum of the space of covariant constant sections in $\Gamma\big(I,a^\ast (TM\oplus L)\big)$, shifted up in degree by multiplying with the formal symbol $dt$ of degree $1$,  and the covariant affine-linear sections in $\Gamma(I,a^\ast L)$. 
The operations $(\mu_k)_{0\leq k<n+1}$ on the graded vector bundle
 (\ref{grvb}) are determined by the transfer process. They can all be made explicit as (finite) sums over rooted trees, and 
 therefore are all smooth operations. Note that as a graded vector bundle,
 $P_\con (TM\oplus L)\,dt\oplus P_\lin L$ is in fact isomorphic to 
\begin{equation}
\ev_0^* (TM\oplus  L)\,dt \oplus \ev_0^* L \oplus \ev_1^*L 
\end{equation}
where $\ev_0, \ \ev_1: P_gM \subset M\times M\to M$ denote the restriction of projection to
the first and second component, respectively.

To prove that 
\begin{items}
\item the object (\ref{grvb}) is, indeed, a vector bundle over the manifold $P_gM$, and
\item the operations $(\mu_k)$ on (\ref{grvb}) are differentiable, 
\end{items}
we only make use of standard facts about bundles, connections, and parallel
 transport since the  formula for the $L_\infty[1]$-algebra operations
 in the transfer theorem is explicit and universal. 
Although one can view the path spaces $PT\M[-1]$ of the AKSZ construction
\cite{MR1432574} as
infinite-dimensional  manifolds  \cite{MR583436},
the smooth structure on   
$PT\M[-1]$ is actually never used.

Some remarks are in order. In \cite{BLX2}, we
 further investigate properties of fibrations of $L_\infty$-bundles.
 As applications of the factorization theorem, we
  prove an inverse function theorem for $L_\infty$-bundles and
study the relation between weak equivalence and quasi-isomorphism
of $L_\infty$-bundles. In addition, we apply the inverse function theorem to prove that the homotopy category of $L_\infty$-bundles has a simple description in terms of homotopy classes of morphisms, when we restrict $L_\infty$-bundles to their germs about their classical loci. 
In a work in progress, we will study derived differentiable stacks (and, in particular, derived orbifolds)
 in terms of Lie groupoids in the category of $L_\infty$-bundles along a similar line as in \cite{MR2817778}. 
It would be interesting to investigate the relation between
derived orbifolds and the Kuranishi spaces of  Fukaya-Oh-Ohta-Ono
 \cite{fukaya2015kuranishi, MR3931096}. 
We expect that derived orbifolds in our sense will provide a useful tool for the study of symplectic reduction. In particular, it would be interesting  to explore
the relation  between ``symplectic derived orbifolds" 
with  the Sjamaar-Lerman theory of stratified symplectic spaces and reduction \cite{MR1127479}. 
It would also be worth investigating the relation between our
$L_\infty$-bundles and Costello's ``$L_\infty$-spaces" \cite{costello2011geometric, MR3358542, MR4079151}. 
After the posting of the first e-print version \cite{2020arXiv200601376B} on
 arXiv.org, Pridham informed us that he can construct a model category for which our category of $L_\infty$-bundles is the subcategory of fibrant objects. (See \cite{MR4036665} for related results.) 
Also, Carchedi and Roytenberg informed us that they made a conjecture to the effect that non-positively graded dg-manifolds form a category of fibrant objects in 2013.

\subsection*{Notations and conventions}

Differentiable means $C^\infty$. Manifold means differentiable manifold, which includes second countable and Hausdorff as part of the definition.  Hence manifolds admit partitions of unity, which implies that vector bundles admit connections, and fiberwise surjective homomorphisms of vector bundles admit sections.

For any graded vector bundle $E$ over a manifold $M$, we sometimes use the same symbol $E$ to denote the sheaf of its sections over $M$, by abuse of notation. 
In particular, for a graded vector bundle $L$ over $M$, by
$\Sym_{\O_M} L^\vee$, or simply $\Sym L^\vee$, we denote the sheaf of sections of the graded vector bundle $\Sym L^\vee$ over $M$, i.e.\ the sheaf of
fiberwise polynomial functions on $L$.

The notation $| \cdot |$ denotes the degree of an element. When we use this notation, we assume the input is a homogeneous element.

We will need two types of connections --- affine connections on a manifold and linear connections on a vector bundle. For a vector bundle $L$ over $M$, we denote an affine connection on $M$ by $\cntnM$ and denote a linear connection on $L$ by $\cntnL$.

\subsection*{Acknowledgments}
We would like to thank several institutions for their hospitality,
 which allowed the completion of this project: Pennsylvania State University
 (Liao),  University of British Columbia (Liao and Xu),
National Center for Theoretical Science  (Liao), KIAS (Xu), Institut des Hautes \'{E}tudes Scientifiques (Liao) and Institut Henri Poincar\'{e} (Behrend, Liao and Xu). 
We also wish to thank  Ruggero Bandiera, 
David Carchedi, 
Alberto Cattaneo, 
 David Favero,
 Domenico Fiorenza, Ezra Getzler,
 Owen Gwilliam,
 Vladimir Hinich,
 Bumsig Kim,
 Marco Manetti, Raja Mehta, Pavel Mnev,
Joost Nuiten, 
Jonathan Pridham, 
 Dima Roytenberg,
Pavol Severa,  Mathieu Sti\'enon and Ted Voronov
 for fruitful discussions and useful comments.

\section{Preliminaries}

\subsection{Categories of fibrant objects}\label{sec:CFO}

The notion of categories of fibrant objects was introduced by Kenneth Brown \cite{MR341469}. 
Below we recall a few basic facts about categories of fibrant objects. 
For details, we refer the reader to \cite{MR341469}.

\begin{defn}\label{def:CFO}
Let $\C$ be a category with finite products and a final object. 
We say $\C$ is a \textit{category of fibrant objects} if it has two distinguished classes of morphisms, called \textit{weak equivalences} and \textit{fibrations}, and satisfies the following axioms:
\begin{items}
\item\label{axiom:2out3} Weak equivalences satisfy {\it two out of three}, i.e. if $f$ and $g$ are two morphisms such that $fg$ is defined, and if two of $f$, $g$, $fg$ are weak equivalences, then the third one must be also a weak equivalence. All isomorphisms are weak equivalences.
\item\label{axiom:Fibration} The composition of two fibrations is a fibration. All isomorphisms are fibrations.
\item\label{axiom:PullbackFib} Every base change (i.e. pull back) of a fibration exists, and is again a fibration. That is, given a diagram $X \xrightarrow{f} Z \xleftarrow{g} Y$ with $g$ a fibration, the fibered product $X \times_Z Y$ exists and the projection $X \times_Z Y \to X$ is a fibration. 
\item\label{axiom:PullbackTriFib} Every base change of a {\it trivial fibration} (i.e. a fibration which is also a weak equivalence) is a trivial fibration.
\item\label{axiom:FibrantObj} Every object is \textit{fibrant} (i.e. the morphism from an object to the final object is a fibration).
\item\label{axiom:Factorization} Every morphism can be factored as a weak equivalence followed by a fibration.
\end{items}
\end{defn}

\begin{ex}\label{ex:CFOcomplex}
Let $\A$ be an abelian category. The category of cochain complexes in $\A$ (infinite in both directions) is a category of fibrant objects. Here weak equivalences are quasi-isomorphisms of complexes, and fibrations are surjective cochain maps. See \cite{MR341469,MR0222093}.
\end{ex}

\begin{defn}
Let $\C$ be a category of fibrant objects. A \textit{path space object} for an object $X$ in $\C$ is an object $P_X$ together with morphisms 
\begin{equation}\label{eq:PathSpDecomp}
\xymatrix{
X \ar[r]^-s & P_X \ar[r]^-{e} & X \times X,
}
\end{equation}
where $s$ is a weak equivalence, $e$ is a fibration, and the composition $e \circ s$ is the diagonal map
\begin{equation}\label{eq:diagonal}
\Delta: X \to X \times X.
\end{equation}
Decomposition~\eqref{eq:PathSpDecomp} will be referred to as a \textit{path space decomposition}.
\end{defn}

Given a category of fibrant objects $\C$, consider its \emph{homotopy category} $\Ho(\C)$ which is obtained by formally inverting weak equivalences.

\begin{lem}\label{lem:UniquePathSp}
Given any object $X$ in a category of fibrant objects $\C$, the decomposition \eqref{eq:PathSpDecomp} is unique in the sense that if $X \to P_X' \to X \times X$ is another path space decomposition, then there exists a third path space object $P_X''$ together with trivial fibrations $P_X \leftarrow P_X'' \to P_X'$ such that the following diagram is commutative.
\begin{equation}\label{eq:UniquePathSp}
\begin{tikzcd}
 & P_X \ar[rd] &  \\
X \ar[ru] \ar[r] \ar[rd] & P_X'' \ar[r] \ar[u] \ar[d] & X \times X  \\
 & P_X' \ar[ru] & 
\end{tikzcd}
\end{equation}
In particular, the path space object $P_X$ for $X$ is unique up to isomorphisms in  the homotopy category $\Ho(\C)$.
\end{lem}
\begin{pf}
Let $X \xto{s} P_X \xto{e} X\times X$ and $X \xto{s'} P_X' \xto{e'} X\times X$ be two path space decompositions for $X$. Let $A = P_X \times_{X \times X} P_X'$. By Axiom~\ref{axiom:Factorization}, there exists a weak equivalence $s'':X \to P_X''$ and a fibration $f:P_X'' \onto A$ such that $f \circ s'' = (s,s'): X \to A$. We denote $f_1 =\pr_1 \circ f: P_X'' \to A \to P_X$ and $f_2 =\pr_2 \circ f: P_X'' \to A \to P_X'$. The morphisms are summarized in the following commutative diagram:
 $$
\begin{tikzcd}
 & P_X  &  \\
X \ar[ru,"s"] \ar[r,"s''"] \ar[rd,"s'"'] & P_X'' \ar[r,two heads,"f"] \ar[u,two heads,"f_1"'] \ar[d,two heads,"f_2"] & A \ar[lu,two heads, "\pr_1"'] \ar[ld,two heads, "\pr_2"] \\
 & P_X'  & 
\end{tikzcd}
$$
Note that since $s = f_1 \circ s''$, it follows from Axiom~\ref{axiom:2out3} that $f_1$ is a trivial fibration. Similarly, so is $f_2$.

Now let $e'' = e \circ f_1:P_X'' \to X \times X$. It is a fibration since so are $e$ and $f_1$. Furthermore, 
$$
e'' \circ s'' = e \circ f_1 \circ s'' = e \circ s = \Delta,
$$
where $\Delta$ is the diagonal map \eqref{eq:diagonal}. 
Since $e \circ \pr_1 = e' \circ \pr_2:A \to X \times X$, we have that $e'' = e\circ f_1 = e' \circ f_2$. Thus, the diagram~\eqref{eq:UniquePathSp} commutes, and the proof is complete.
\end{pf}

In fact, the existence of path space objects is equivalent to Axiom~\ref{axiom:Factorization} provided all the other axioms in Definition~\ref{def:CFO} hold. 

\begin{lem}{\bf(\cite[Factorization lemma]{MR341469})}
Let $\C$ be a category which satisfies all the axioms in Definition~\ref{def:CFO} except \ref{axiom:Factorization}. If every object in $\C$ has a path space object, then $\C$ is a category of fibrant objects. 
\end{lem}

For given arbitrary morphisms $X \to Z \leftarrow Y$  in $\C$, one can form a \emph{homotopy fibered product} $X \times_Z^h Y$ in the homotopy category $\Ho(\C)$.  If $P_Z$ is a path space object for $Z$, then $X \times_Z P_Z \times_Z Y$ is a representative of the homotopy fibered product $X \times_Z^h Y$. 
The following lemma is standard for categories of fibrant objects. We sketch a proof for completeness.

\begin{lem}
Let $\C$ be a category of fibrant objects. The isomorphism class of $X \times_Z P_Z \times_Z Y$ in the homotopy category $\Ho(\C)$ is independent of the choice of the path space object $P_Z$. 
\end{lem}
\begin{pf}
Let $Z \xto{s} P_Z \xto{e} Z\times Z$ and $Z \xto{s'} P_Z' \xto{e'} Z\times Z$ be two path space decompositions for $Z$.  
By Lemma~\ref{lem:UniquePathSp}, we may assume there is a trivial fibration $f:P_Z' \to P_Z$ such that $s = f \circ s'$ and $e' = e \circ f$. 

We pull back $f:P_Z' \to P_Z$ along $X \times_Z P_Z \to P_Z$ and obtain a trivial fibration $t : X \times_Z P_Z'  \to X \times_{Z} P_Z.$ 
Since $e' = e \circ f$, we have the commutative diagram:
$$
\begin{tikzcd}
X \times_Z P_Z' \ar[d,two heads,"t"'] \ar[r] & P_Z' \ar[d,two heads,"f"'] \ar[rd,"\ev_1'"] &  \\
X \times_Z P_Z \ar[r] & P_Z \ar[r,"\ev_1"'] & Z
\end{tikzcd}
$$
Thus,  
$$
X \times_Z P_Z' \times_Z Y \cong (X \times_Z P_Z') \times_{X \times_{Z} P_Z} ((X \times_Z P_Z) \times_Z Y). 
$$
Consequently, by pulling back $t$, we obtain another trivial fibration $t':X \times_Z P_Z' \times_Z Y \to X \times_Z P_Z \times_Z Y$
$$
\begin{tikzcd}
X \times_Z P_Z' \ar[d,two heads,"t"']  & \ar[l] X \times_Z P_Z' \times_Z Y \ar[d,two heads,"t'"]   \\
X \times_Z P_Z  & \ar[l] X \times_Z P_Z \times_Z Y,
\end{tikzcd}
$$ 
and hence the two representatives $X \times_Z P_Z' \times_Z Y$ and $ X \times_Z P_Z \times_Z Y$ are isomorphic in $\Ho(\C)$.
\end{pf}

\subsection{Curved $L_\infty[1]$-algebras}\label{linfty}

We will briefly review the notion of (positively graded) curved $L_\infty[1]$-algebras. We refer the reader to \cite{MR2931635,MR1235010,MR1327129,MR1183483} for a general introduction to $L_\infty$-algebras.

Let $L$ and $E$ be graded vector spaces, concentrated in finitely many positive degrees. 
Following \cite{ezra}, we denote by $\Sym^{n,i}(L,E)$, for $n\geq0$,
the vector space of graded symmetric multilinear maps 
$$\underbrace{L\times\ldots\times L}_{n}\longrightarrow E$$
of degree $i$. 
Note that $\Sym^{n,i}(L,E)$ vanishes, as soon as $n+i$ exceeds the
highest degree of $E$. We write
$$\Sym^i(L,E)=\bigoplus_{n\geq0}\Sym^{n,i}(L,E)\,.$$
There is a binary operation 
$$\Sym^i(L,E)\times \Sym^j(L,L)\longrightarrow \Sym^{i+j}(L,E)\,,$$
defined by 
$$(\lambda\circ\mu)_n(x_1,\ldots,x_n)=
\sum_{\sigma\in
  S_n}(-1)^\epsilon\sum_{k=0}^n\frac{1}{k!(n-k)!}\,
\lambda_{n+1-k}\big(\mu_k(x_{\sigma(1)}\ldots),
\ldots,x_{\sigma(n)}\big)\,,$$ 
for $x_1,\ldots,x_n$ elements of $L$. Here $(-1)^\epsilon$ is determined by Koszul sign convention. This operation is linear in each argument. It is not associative, but it does satisfy the pre-Lie
algebra axiom.  In particular, taking $E=L$, the graded commutator with respect to $\circ$ defines the structure of a graded Lie algebra on 
$$\Sym(L,L)=\bigoplus_{i}\Sym^i(L,L)\,.$$

If $\delta$ is a differential in $L$, then $[\delta,\argument]$
is an induced differential on $\Sym(L,L)$ (acting by derivations with respect to the bracket), turning $\Sym(L,L)$ into a differential graded Lie algebra. 

\begin{defn} 
Given a complex $(L,\delta)$,  the structure of a {\it curved $L_\infty[1]$-algebra} on $(L,\delta)$ is a Maurer-Cartan element in $\Sym(L,L)$, i.e. an element $\lambda$ of  $\Sym^1(L,L)$, satisfying the Maurer-Cartan equation
\begin{equation}
[\delta,\lambda]+\lambda\circ\lambda=0\,.
\end{equation}
\end{defn}

The case where $\delta=0$ is no less general than the case of arbitrary $\delta$, because we can always incorporate $\delta$ into $\lambda$ by adding it to $\lambda_1$.  In the case where $\delta=0$, the Maurer-Cartan equation reduces to 
$$\lambda\circ\lambda=0\,.$$
Nevertheless, for the transfer theorem, the case of general $\delta$ will be important to us.

There is another binary operation
$$\Sym^i(L,L)\times\Sym^0(E,L)\longrightarrow \Sym^i(E,L)\,,$$
defined by 
$$
(\lambda\bullet\phi)_n(x_1,\ldots,x_n)
=\sum_{\sigma\in S_n}(-1)^\epsilon 
\sum_{k=0}^n\frac{1}{k!}\sum_{n_1+\ldots+n_k=n}\frac{1}{n_1!\ldots n_k!}\lambda_k\big(\phi_{n_1}(x_{\sigma(1)},\ldots),\dots,\phi_{n_k}(\ldots,x_{\sigma(n)})\big)\,.
$$
This operation is linear only in the first argument, but it is associative. 

\begin{defn}
If $\lambda$ and $\mu$ are curved $L_\infty[1]$-structures on $(L,\delta)$ and $(E,\delta')$, respectively, and $\phi\in \Sym^0(L,E)$, then $\phi$ is a {\it morphism of curved $L_\infty[1]$-algebras} if 
\begin{equation}\label{eq:LooMor}
\phi\circ(\delta+\lambda)=(\delta'+\mu)\bullet\phi\,.
\end{equation}
\end{defn}

It is straightforward to show that any composition of morphisms of curved $L_\infty[1]$-algebras is another morphism of curved $L_\infty[1]$-algebras. See \cite{ezra} for details.

For example, 
if $\lambda$ is a curved $L_\infty[1]$-structure on $(L,\delta)$, and $E\subset L$ is a subcomplex preserved by all operations $\lambda$, then the inclusion $E\to L$, without any higher correction terms, is a morphisms of curved $L_\infty[1]$-algebras. We say that $E$ is a curved $L_\infty[1]$-subalgebra of $L$. 

\subsection{The transfer theorem for curved $L_\infty[1]$-algebras}\label{tranthem}

Let $(L,\delta)$  be
 a cochain complex of vector spaces.  Assume that $L$ is concentrated in finitely many positive degrees. Assume  also  given a descending filtration 
$$L=F_0L\supset F_1L\supset\ldots$$
on $L$, such that both the  grading and the differential $\delta$ are compatible with $F$. 
We assume that   $F_k L=0$, for $k\gg0$.

\begin{ex}
A {\em natural filtration }of $L$ is given by 
$$F_k L=\bigoplus_{i\geq k}L^i\,.$$
Every $L_\infty[1]$-structure on $L$ is filtered with respect to this
natural filtration. 
\end{ex}

Let $\eta$ be a linear operator of degree $-1$ on $L$, which is compatible with the filtration $F$, and satisfies the two conditions
$$\eta^2=0\,,\qquad \eta\delta\eta=\eta\,.$$
(We call $\eta$ a {\em contraction} of $\delta$.)

Under these hypotheses, $\delta\eta$ and $\eta\delta$ are idempotent operators on $L$.  We define $$H=\ker[\delta,\eta]=\ker(\delta\eta)\cap\ker(\eta\delta)=\im(\id_L-[\delta,\eta])\,.$$  
Let us write $\iota:H\to L$ for the inclusion, and $\pi:L\to H$ for the projection. We have 
$$\iota \pi=\id_L-[\delta,\eta]\,, \quad\text{and}\quad \pi\iota=\id_H\,.$$ 
We write $\delta$ also for the induced differential on $H$. Then $\iota:(H,\delta)\to (L,\delta)$ is a homotopy equivalence with homotopy inverse $\pi$. 

Let $\lambda=(\lambda_k)_{k\geq0}$ be a curved $L_\infty[1]$-structure in $(L,\delta)$.  We assume that this $L_\infty[1]$-structure is {\em nilpotent}, 
i.e.\ that  $\lambda$  increases the filtration degree by 1. This means that, for $n\geq1$, we have  $\lambda_n(F_{k_1},\ldots,F_{k_n})\subset F_{k_1+\ldots+k_n+1}$.  It puts no restriction on $\lambda_0$.  (This convention on $\lambda_0$ is different from \cite{MR1950958,ezra}.  In \cite{MR1950958,ezra}, it is assumed that the image of $\lambda_0$ is contained in $F_1L^1$.  We do not need this stronger requirement, because all our $L_\infty[1]$-algebras are in bounded degree.)

\begin{prop}[Transfer Theorem]\label{transfertheorem}
There is a unique  $\phi\in \Sym^0(H,L)$ satisfying the equation
\begin{equation}\label{recu}
\phi=\iota-\eta \lambda\bullet\phi\,.
\end{equation}
Setting
$$\mu=\pi\lambda\bullet\phi\in \Sym^1(H,H)$$ defines a curved $L_\infty[1]$-structure on $(H,\delta)$, such that $\phi$ is a morphism of curved $L_\infty[1]$-structures from $(H,\delta+\mu)$ to $(L,\delta+\lambda)$. 

Furthermore, there exists a morphism of curved $L_\infty[1]$-algebras $\tilde\pi:(L, \delta + \lambda) \to (H, \delta + \nu)$ satisfying the equations 
\begin{align}
& \tilde\pi_1 = \pi - \tilde\pi_1 \lambda_1 \eta, \\
& \tilde\pi \bullet \phi = \id_H.
\end{align}
\end{prop}
\begin{proof}
The proof in \cite{ezra} applies to our situation. 
Our assumptions imply that the number of internal nodes (i.e.\ nodes of valence at least 2)  is bounded above by the highest filtration degree in $L$, and the valence of all  nodes is bounded by the highest degree in $L$. This bounds the number of trees itself, and hence the number of transferred operations.

For the last part, since the curved $L_\infty[1]$-algebras are assumed to be bounded positively graded and nilpotent, one may modify the  approach in \cite{MR3276839} to obtain the projection morphism $\tilde \pi$. 
\end{proof}

\begin{rmk}\label{rmk:TreeFormula} 
One can find $\phi:H\to L$ (and $\mu$) explicitly by solving (\ref{recu}) recursively.  The result is as follows.

Consider  rooted trees with a positive number $n$ of  leaves, and a non-negative number of nodes. For every such tree,
 define an element of $\Sym^{n,0}(H,L)$ by decorating each leaf (an edge without a node at the top) with $\iota$, each node with the applicable $\lambda_k$, and all edges other than leaves with $\eta$. See Figure~\eqref{fig:TreeForIncl} for an example.  (So the root is labeled with $\eta$, unless the root is a leaf, which happens only for the tree without any nodes as in Figure~\eqref{fig:TreeWithoutNodes}.) Then $\phi$ is equal to the sum over all such trees of the corresponding element of $\Sym^0(H,L)$ with proper coefficients. Since $\lambda$ is nilpotent, only finitely many trees give non-zero contributions.

\begin{figure}[h]
\begin{subfigure}{.5\textwidth}
\centering
\begin{tikzpicture}[
    sibling distance  = 1.5cm,
    level distance  = 1cm,
    grow=up
  ]
\node {}
child {
node[circle, draw]{$\lambda_3$}
child {
node[circle, draw, black]{$\lambda_2$}
child {
node {}
edge from parent[black]
node[right]{ $\iota$}
}
child {
node {}
edge from parent[black]
node[left] {$\iota$}
}
edge from parent[blue]
node[right, near start] {\; $\eta$}
}
child {
node {}
edge from parent[black]
node[left] {$\iota$}
}
child {
node[circle, draw, black] {$\lambda_1$}
child {
node {}
edge from parent[black]
node[left] {$\iota$}
}
edge from parent[blue]
node[left, near start] {$\eta$ \;}
}
edge from parent[red]
node[right] {$\eta$}
}
;
\end{tikzpicture}
\caption{A tree for inclusion morphism}
  \label{fig:TreeForIncl}
\end{subfigure}
\begin{subfigure}{.5\textwidth}
\centering
\begin{tikzpicture}[
    sibling distance  = 1.5cm,
    level distance  = 1cm,
    grow=up,
  ]
\node {}
child {
edge from parent[red]
node[right] {$\iota$}
}
;
\end{tikzpicture}
\caption{The tree without nodes}
  \label{fig:TreeWithoutNodes}
\end{subfigure}
\caption{Decorated trees in the formulas of inclusion morphism $\phi$}
\end{figure}
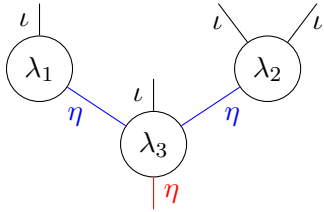
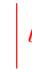

Consider also rooted trees with a non-negative number $n$ of leaves, and a positive number of nodes. For every tree of this kind, define an element of $\Sym^{n,1}(H,H)$ by decorating each leaf with $\iota$, each node with the applicable $\lambda_k$, all internal edges with $\eta$, and the root with $\pi$. See Figure~\eqref{fig:TreeForOp} for an example. Then $\mu$ is equal to the sum over all such trees of the corresponding element of $\Sym^1(H,H)$ with proper coefficients. Again, only finitely many trees contribute to the sum.

A priori, one may have a tree with $\lambda_0$ labeled at a node such as Figure~\eqref{fig:VanishingTree}. Nevertheless, it follows from our degree assumptions that $\eta(\lambda_0)=0$, and thus the terms associated with such trees vanish except the transferred curvature $\pi(\lambda_0)$. 

In the homotopy transfer theorem for non-curved $L_\infty[1]$-algebras \cite{2017arXiv170502880B,FioMan},  the cochain differential $\delta$ is assumed to be the first bracket, and $\lambda$ consists of the operations with two or more inputs. Thus, the nodes in the tree formulas \cite{FioMan} have at least two ascending edges, while the nodes in our trees may have only one ascending edge as in Figure~\eqref{fig:TreeForOp}. 
\end{rmk}

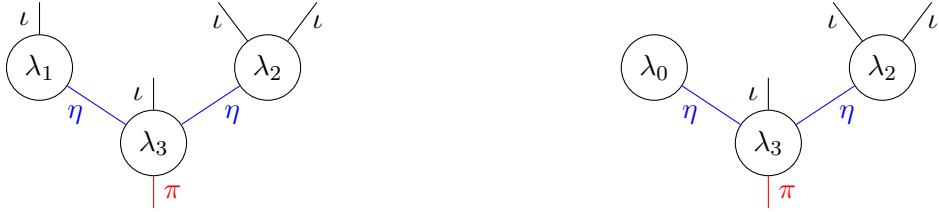
\begin{figure}[h]
\begin{subfigure}{.5\textwidth}
\centering
\begin{tikzpicture}[
    sibling distance  = 1.5cm,
    level distance  = 1cm,
    grow=up
  ]
\node {}
child {
node[circle, draw]{$\lambda_3$}
child {
node[circle, draw, black]{$\lambda_2$}
child {
node {}
edge from parent[black]
node[right]{ $\iota$}
}
child {
node {}
edge from parent[black]
node[left] {$\iota$}
}
edge from parent[blue]
node[right, near start] {\; $\eta$}
}
child {
node {}
edge from parent[black]
node[left] {$\iota$}
}
child {
node[circle, draw, black] {$\lambda_1$}
child {
node {}
edge from parent[black]
node[left] {$\iota$}
}
edge from parent[blue]
node[left, near start] {$\eta$ \;}
}
edge from parent[red]
node[right] {$\pi$}
}
;
\end{tikzpicture}
\caption{A tree for transferred operations}
  \label{fig:TreeForOp}
\end{subfigure}
\begin{subfigure}{.5\textwidth}
\centering
\begin{tikzpicture}[
    sibling distance  = 1.5cm,
    level distance  = 1cm,
    grow=up,
  ]
\node {}
child {
node[circle, draw]{$\lambda_3$}
child {
node[circle, draw, black]{$\lambda_2$}
child {
node {}
edge from parent[black]
node[right]{ $\iota$}
}
child {
node {}
edge from parent[black]
node[left] {$\iota$}
}
edge from parent[blue]
node[right, near start]{\; $\eta$}
}
child {
node {}
edge from parent[black]
node[left] {$\iota$}
}
child {
node[circle, draw, black]{$\lambda_0$}
edge from parent[blue]
node[left, near start] {$\eta$ \;}
}
edge from parent[red]
node[right] {$\pi$}
}
;
\end{tikzpicture}
\caption{An ignorable tree}
  \label{fig:VanishingTree}
\end{subfigure}
\caption{Decorated trees in the formulas of transferred operations}
\end{figure}

\begin{rmk}
Consider the case $\lambda=\lambda_1$. Then   $\eta\lambda$ and $\lambda\eta$ are nilpotent linear operators on $L$.  The transfer theorem yields
$$\phi=\phi_1=(1+\eta\lambda)^{-1}\iota\,,$$
and
$$\mu=\mu_1=\pi\lambda(1+\eta\lambda)^{-1}\iota=\pi(1+\lambda\eta)^{-1}\lambda\iota\,.$$
In this case we have more: if we define 
$$\tilde\pi=\pi(1+\lambda\eta)^{-1}\quad\text{and}\quad
\tilde\eta=\eta(1+\lambda\eta)^{-1}=(1+\eta\lambda)^{-1}\eta\,,$$
we obtain a deformation of the `contraction data' given by the pair $(\delta,\eta)$. See, for example, \cite[Appendix~A]{MR4665716}.  In fact, $\phi:(H,\delta+\mu)\to(L,\delta+\lambda)$ and $\tilde\pi:(L,\delta+\lambda)\to (H, \delta+\mu)$ are homomorphisms of complexes, $\tilde\pi\phi=\id_H$, and $\phi\tilde\pi=\id_L-[\delta+\lambda,\tilde\eta]$. This is a special case of the homological perturbation lemma \cite{MR0220273,2004math......3266C}.
\end{rmk}

\begin{rmk}\label{rmk:0th1stMapInHptTransfer}
In the general case, we always have
\begin{align*}
\phi_1 & =(1+\eta\lambda_1)^{-1}\iota\,, \\
\mu_0 & =\pi(\lambda_0)\, , \\
\mu_1 &=\pi\lambda_1(1+\eta\lambda_1)^{-1} \iota=\pi(1+\lambda_1\eta)^{-1}\lambda_1 \iota \,, \\
\tilde\pi_1 & = \pi (1+\lambda_1 \eta)^{-1} \, .
\end{align*}
\end{rmk}

\section{The category of $L_\infty$-bundles}

\subsection{$L_\infty$-bundles}
\begin{defn}
An {\it $L_\infty$-bundle} is a triple $\M=(M,L,\lambda)$, where $M$ is
a manifold (called the base), $L$ is a (finite-dimensional) graded vector bundle
$$L=L^1\oplus\ldots\oplus L^{n}$$ over $M$, and 
$\lambda=(\lambda_k)_{k\geq0}$ is a sequence of 
multi-linear operations
of degree $+1$
\begin{equation}
\label{eq:SCE}
\lambda_k:\underbrace{ L\times_M\ldots \times_ML}_k\longrightarrow L
\,,\qquad k\geq0\,.
\end{equation}
The operations $\lambda_k$ are required to be differentiable maps over $M$, and to make each fiber $(L,\lambda)|_P$,
for $P\in M$, into a curved $L_\infty[1]$-algebra.
\end{defn}
  We say that $(L,\lambda)$ is
a {\it bundle of curved $L_\infty[1]$-algebras }over $M$. 
 By our assumption on the
degrees of $L$,  we have $\lambda_k=0$, for $k\geq n$. The integer $n$
is called the {\it amplitude }of $\M$. 

An $L_\infty$-bundle is called {\it quasi-smooth} if its amplitude is one.
 That is, a quasi-smooth $L_\infty$-bundle consists of
 a triple $(M,L,\lambda)$,
 where $L = L^1$ is a vector bundle of degree $1$ over $M$,
 and $\lambda = \lambda_0$ is a global section of $L$.
Such an  $L_\infty$-bundle can be considered as the ``derived intersection"
of $\lambda_0$ with the zero section of $L$.
 If $f$ is a smooth function on $M$, the quasi-smooth $L_\infty$-bundle
 $(M, T^\vee M[-1], df)$ can be thought  as the ``derived critical locus'' of $f$.

Recall that the curved $L_\infty[1]$-axioms for the $\lambda = (\lambda_k)_{k \geq 0}$ can be summarized in the single equation
$$\lambda\circ\lambda=0\,.$$
(See Section~\ref{linfty} for the notation.) 
The axioms are enumerated by the number of arguments they take.  The first few are 
\begin{items}
\item ($n=0$) $\lambda_1(\lambda_0)=0$,
\item ($n=1$) $\lambda_2(\lambda_0,x)+\lambda_1^2(x)=0$.
\item ($n=2$) $\lambda_3(\lambda_0,x,y)+\lambda_2(\lambda_1(x),y)+ (-1)^{|x||y|}\, \lambda_2(\lambda_1(y),x)+\lambda_1(\lambda_2(x,y))=0$. 
\end{items}
If all $\lambda_k$, for $k\geq3$, vanish, then $L[-1]$ is a bundle of curved differential graded Lie algebras over $M$. Guided by this correspondence, we sometimes call the section $\lambda_0$ of $L^1$ over $M$ the {\it curvature }of $\M$, call  $\lambda_1:L\to L[1]$ the {\it twisted differential}, and $\lambda_2$ the {\it bracket}.  The $\lambda_k$, for $k\geq3$, are the {\it higher brackets}.

\begin{defn}
The set of points $Z(\lambda_0)\subset M$ where the curvature vanishes is a closed subset of $M$.  It is called the {\it classical} or
  {\it Maurer-Cartan locus }of the $L_\infty$-bundle $\M$, notation
$\pi_0(\M)$. We consider $\pi_0(\M)$ as a set without any further
structure.
\end{defn}

\begin{defn}\label{defnlin}
A {\it morphism }of $L_\infty$-bundles $(f,\phi):(M,L,\lambda)\to
(M',L',\lambda')$ consists of a differentiable map $f:M\to M'$, and 
a sequence of operations  
$$\phi_k:\underbrace{L\times_M\ldots\times_M L}_k\longrightarrow  L'\,,\qquad k\geq1\,,$$
of degree zero. The operations $\phi_k$ are required to be fiberwise multilinear differentiable maps covering $f:M\to M'$, such that for every $P\in M$, the induced
sequence of operations $\phi|_P$ defines a morphism of curved $L_\infty[1]$-algebras $(L,\lambda)|_P\to (L',\lambda')|_{f(P)}$. 
A morphism $(f,\phi):(M,L,\lambda)\to (M',L',\lambda')$ of $L_\infty$-bundles is said to be {\it linear}, if $\phi_k=0$, for all $k>1$. 
\end{defn}

Observe that, for a morphism of $L_\infty$-bundles $(f,\phi):(M,L,\lambda) \to (M',L',\lambda')$, the first map $\phi_1:L\to f^\ast L'$ is a morphism of graded vector
bundles. 
Also note that our degree assumptions imply that $\phi_k=0$, for
$k>n'$, where $n'$ is the amplitude of $(M',L',\lambda')$.

\begin{rmk}
The linear morphisms define a subcategory of the category of $L_\infty$-bundles, containing all identities, although not all isomorphisms.
\end{rmk}

Recall that the $L_\infty[1]$-morphism axioms
 can be summarized in the single equation
(see Section~\ref{linfty} for the notations) 
$$\phi\circ\lambda=\lambda'\bullet\phi\,.$$
They are enumerated by the number of
arguments they take.  The first few are
\begin{items}
\item ($n=0$) $\phi_1(\lambda_0)=\lambda_0'$,
\item ($n=1$) $\phi_2(\lambda_0,x)+\phi_1(\lambda_1(x))=\lambda'_1(\phi_1(x))$.
\end{items}
The first property implies that a morphism of $L_\infty$-bundles induces a
map on classical loci. 

It is possible to compose morphisms of $L_\infty$-bundles, and so we have the
category of $L_\infty$-bundles.

\begin{rmk}\label{finalremark}
The category of $L_\infty$-bundles has a final object.  This is the
singleton manifold with the zero graded vector bundle on it, notation $\ast$.  The
classical locus of an $L_\infty$-bundle $\M$ is equal to the set of morphisms
from $\ast$ to $\M$.
\end{rmk}

\begin{rmk}
If $(M,L,\lambda)$ is an $L_\infty$-bundle, and $f:N\to M$ a differentiable
map, then $(f^\ast L,f^\ast\lambda)$, is a bundle of curved
$L_\infty [1]$-algebras over $N$, enhancing $N$ to an $L_\infty$-bundle $(N,f^\ast L,f^\ast\lambda)$, 
together with a morphism of
 $L_\infty$-bundles $(N,f^\ast L,f^\ast\lambda)\to(M,L,\lambda)$. 
\end{rmk}

\begin{rmk}
If $(M,L,\lambda)$ is an $L_\infty$-bundle, then the set of global sections
$\Gamma(M,L)$ is a curved $L_\infty[1]$-algebra. For every point $P\in M$, evaluation at $P$ defines a linear morphism of $L_\infty[1]$-algebras
 $\Gamma(M,L)\to L|_P$. 
\end{rmk}

\begin{prop}\label{oneiso}
Let $(f,\phi):(M,L,\lambda)\to (M',L',\lambda')$ be a morphism of $L_\infty$-bundles.  If $f:M\to M'$ is a diffeomorphism of manifolds, and $\phi_1:L\to L'$ an isomorphism of vector bundles, then $(f,\phi)$ is an isomorphism of $L_\infty$-bundles. 
\end{prop}
\begin{pf}
Without loss of generality, we assume $M=M'$, $L=L'$, and $f=\id_M$, $\phi_1=\id_L$. First, note that for every sequence of operations $\chi$ on $L$, the equation $\psi\bullet \phi=\chi$ has a unique solution for $\psi$.  In fact, the equation $\psi\bullet\phi=\chi$ defines $\psi_k$ recursively.  Then define $\psi$ such that $\psi\bullet\phi=\id$, where $\id_1=\id_L$, and $\id_k=0$, for $k>1$.  Then $(\phi\bullet\psi)\bullet\phi=\id\bullet\phi$ by associativity of $\bullet$. By the uniqueness result proved earlier, this implies $\phi\bullet\psi=\id$.  
Therefore, $\psi$ is an inverse to $\phi$.  One checks 
that $\lambda'\bullet\phi=\phi\circ\lambda$ implies $\lambda\bullet\psi=\psi\circ\lambda'$, and hence $\psi$ is an inverse for $\phi$ as a morphism of 
$L_\infty$-bundles. 
\end{pf}

\subsection{Algebra model}

For a manifold $M$, we denote the sheaf of $\rr$-valued $C^\infty$-functions by $\O_M$.

\begin{defn}
A {\it graded manifold} $\M$ of {\it amplitude} $[m, n]$
is a pair  $(M,\A)$, where $M$ is a manifold,
$$\A=\bigoplus_{i}\A^i$$
is a sheaf of $\zz$-graded commutative $\O_M$-algebras over (the underlying
topological space) $M$, such that there exists a  $\zz$-graded
vector space  
$$V=V^{m}\oplus V^{m+1} \oplus \cdots V^{n-1}\oplus V^{n}$$  
and a covering of $M$ by open submanifolds $U\subset M$, and
for every $U$ in the covering family, we have 
$$\A|_U\cong C^\infty (U)\otimes \Sym V^\vee\,,$$ 
as sheaf of graded $\O_M$-algebras.
Here $m$ and $n$ may not be non-negative integers, and
we only assume that $m\leq n$.
\end{defn}

\begin{rmk}
In our definition, a graded manifold $\M$ has two types of local coordinates of degree zero --- one is $C^\infty$ coordinates from the base manifold $M$, and the other one is polynomial coordinates from $V^0$. By amplitude, we mean the degree amplitude of $V$ (ignoring the $C^\infty$ part).
\end{rmk}

We say that a graded manifold $\M$ is finite-dimensional
if $\dim M$ and $\dim V$ are both finite. 
We refer the reader to~\cite[Chapter~2]{MR2709144} and~\cite{MR2275685,MR2819233}
for a short introduction to graded manifolds and relevant references.
By $\Gamma (M, \A)$, we denote the space of global sections.

\begin{defn} 
A {\it differential graded manifold} (or {\it dg manifold}) is a triple $(M,\A,Q)$, where $(M,\A)$ is
a graded manifold, and $Q:\A\to \A$
is a  degree 1 derivation of $\A$ 
as sheaf of $\rr$-algebras, such that $[Q,Q]=0$.
\end{defn}

Since we  are only interested  in the $C^\infty$-context,  in
the above definition, the existence of such  $Q:\A\to \A$ is 
equivalent to the existence of  degree 1 derivation on the global sections
$$Q:  \Gamma({M, \A})\to \Gamma({M, \A})$$
such that  $[Q,Q]=0$. Thus $Q$ can be considered as
a vector field on $\M$, notation $Q\in \XX (\M)$,
called a  {\it homological vector field} on the  graded manifold $\M$
 \cite{MR1432574,MR2819233}.

\begin{defn}
A {\it morphism }of dg manifolds $(M,\A,Q)\to (N,\B,Q')$ is a pair
$(f,\Phi)$, where $f:M\to N$ is a differentiable map of manifolds, and
$\Phi:\B\to f_\ast \A$ (or equivalently $\Phi:f^\ast \B\to \A$) is a
morphism of sheaves of graded  algebras, such that $Q\,\Phi=\Phi\, f^\ast Q'$.
\end{defn}

This defines the category of dg manifolds.

Let $\M=(M,L,\lambda)$ be an $L_\infty$-bundle. 
The sheaf of graded algebras $\A=\Sym_{\O_M}L^\vee$ is locally free
 (on generators in negative bounded degree), and thus
$(M, \A)$ is a graded manifold of  amplitude  $[1, n]$, called
{\it  graded manifolds of positive amplitude} by abuse of notation.
  The sum of the duals of the $\lambda_k$ defines a homomorphism of
 $\O_M$-modules $L^\vee\to \Sym_{\O_M}L^\vee$, which extends, in a unique way, to a derivation $Q_\lambda:\A\to \A$ of degree $+1$. 
 The condition
 $\lambda\circ\lambda=0$ is equivalent to the condition $[Q_\lambda, Q_\lambda]=0$. 

Given a morphism of $L_\infty$-bundles $(f,\phi):(M,L,\lambda)\to (N,E,\mu)$, the sum of the duals of the $\phi_k$  gives rise to a homomorphism of $\O_M$-modules $f^\ast E^\vee\to \Sym_{\O_M} L^\vee$, which extends uniquely to a morphism of sheaves of $\O_M$-algebras $\Phi:f^\ast\Sym_{\O_N}E^\vee\to \Sym_{\O_M}L^\vee$.  Conversely, every morphism of $\O_M$-algebras $f^\ast\Sym_{\O_N}E^\vee\to \Sym_{\O_M}L^\vee$ arises in this way from a unique family of operations $(\phi_k)$. The condition $\phi\circ\lambda=\mu\bullet\phi$ 
is equivalent to $Q_\lambda\Phi=\Phi f^\ast Q_\mu$. 

These considerations show that we have a functor
\begin{align}\label{eq3}
(\text{$L_\infty$-bundles})&\longrightarrow (\text{dg manifolds of
positive amplitude})\\
(M,L,\lambda)&\longmapsto (M,\Sym_{\O_M} L^\vee,Q_\lambda)\,.\nonumber
\end{align}
which is fully faithful.

\begin{prop}
\label{pro:Batchelor}
The functor (\ref{eq3}) is an equivalence of categories. 
\end{prop}
\begin{pf}
It remains to show essential surjectivity. We need to prove that if $\A$ is a
  sheaf of algebras, locally free on generators of negative degree  $\geq -n$,
 then there exist a graded vector bundle $E=E^{-n}\oplus\ldots\oplus E^{-1}$,
 and an embedding $E\hookrightarrow \A$,  such that $\Sym_{\O_M} E\iso\A$.
  Let $\B\subset\A$ be the sheaf of subalgebras generated by all local sections of degree $\geq -n+1$. Then the $\O_M$-module
 $E^{-n}=\A^{-n}/\B^{-n}$ is locally free, and the projection
 $\A^{-n}\to E^{-n}$ admits a section $E^{-n}\hookrightarrow\A^{-n}$. 
 By induction, we have $E^{-n+1}\oplus\ldots\oplus E^{-1}\hookrightarrow \A$, 
inducing an isomorphism $\Sym_{\O_M}(E^{-n+1}\oplus\ldots\oplus E^{-1})\iso \B$.  It follows that $\Sym_{\O_M}(E^{-n}\oplus\ldots\oplus E^{-1})\iso \A$.
Let  $L^{i}=(E^{-i})^\vee$, $i=1, \ldots, n$,
 and $L=L^1\oplus\ldots\oplus L^{n}$. Thus $\A\cong \Sym_{\O_M} L^\vee$.

Moreover,  as the morphism $\O_M\to\A$ has image in degree $0$, for degree
reason, we have $Q(\O_M)=0$, so that $(\O_M, 0)\to (\A,Q)$ is a morphism of 
sheaves of differential graded algebras. It thus follows that the
homological vector field $Q$ is indeed tangent to the fibers of
the graded vector bundle $L$.
Therefore, $Q$ induces a unique family of fiberwise
 operations $(\lambda_k)_{k \geq 0}$ of degree 1 as in  \eqref{eq:SCE}.
 The condition  $[Q, Q]=0$ implies that
 $\lambda\circ\lambda=0$. Hence $(L,\lambda)$ is
a  bundle of curved $L_\infty[1]$-algebras over $M$, and
$(M, L, \lambda)$ is an $L_\infty$-bundle.
\end{pf}

In the supermanifold (i.e.\ $\zz_2$-graded with $Q$ being zero) case,
Proposition~\ref{pro:Batchelor} is known as Batchelor's
theorem --- see~\cite{MR536951,MR2275685,QFT_2-vol_book,2021arXiv210813496K}.

\begin{rmk}
\label{rmk:dg}
Thus, $L_\infty$-bundles, according to our definition, are differential
 graded manifolds endowed with  a generating graded vector bundle of
 positive degrees. 
Some constructions are easier in the algebra model, but many use the graded vector bundle $L$. Of course, a major consideration is that $L$, being finite-rank, is more geometrical than $\A$.
\end{rmk}

\begin{rmk} 	
Dg manifolds of amplitude $[-k, -1]$ with $k \geq 1$ can be thought of as
 \emph{Lie $k$-algebroids} \cite{MR2521116,  MR2441255,MR4007376, MR2223155, MR3090103, MR2768006, MR1958835} 
and \cite[Letters~7 and~8]{2017arXiv170700265S}. 
They can be considered as the infinitesimal counterparts of higher groupoids.

Hence a general dg manifold of amplitude $[m, n]$  
can encode both stacky and derived singularities in
differential geometry.
\end{rmk}

\subsection{\'{E}tale morphisms and weak equivalences}\label{sec:Etale}

Let $\M=(M,L,\lambda)$ be an $L_\infty$-bundle, with curvature $\lambda_0$.  
Let $P\in Z(\lambda_0)$ be a classical point of $\M$. 
Since $\lambda_0(P)=0\in L^1|_P$, there is a natural decomposition
$T L^1|_{\lambda_0(P)} \cong TM|_P \oplus L^1|_P$. By  $D_P \lambda_0$, we denote
the composition
\begin{equation}
\label{eq:Lyon}
D_P \lambda_0:  T M |_P \xxto{T\lambda_0 |_P}T L^1|_{\lambda_0 (P)} \cong T M|_P \oplus L^1|_P \xxto{\pr}
 {L^1}|_P
\end{equation}
where $T\lambda_0 |_P$ denotes the tangent map of $\lambda_0: M\to L^1$ at $P\in M$.
Thus $D_P\lambda_0 :TM|_P\to L^1|_P$ is a linear map, called  the {\em  derivative of the curvature $\lambda_0$}.

\begin{defn}
\label{def:tangentcomplex}
Let $\M=(M,L,\lambda)$ be an $L_\infty$-bundle, with curvature $\lambda_0$ and twisted differential $d=\lambda_1$. Let $P\in
Z(\lambda_0)$ be a classical point of $\M$. 
 We add the derivative of the curvature  $D_P\lambda_0$ to $d|_P$ in the
graded vector space $TM|_P\oplus L|_P$ to define the {\it tangent
  complex }of $\M$ at $P$:
\begin{equation}\label{eq:tangent}
T\M|_P\quad :=\quad\xymatrix{
TM|_P\rto^{D_P\lambda_0}& L^1|_P\rto^{d|_P}& L^2|_P\rto^{d|_P} & \ldots}
\end{equation}
(Note that $TM|_P$ is in degree $0$.)
\end{defn}

  The fact that $d\lambda_0=0$, and that $P$ is a Maurer-Cartan point,   implies that we have a complex.  (Note that we do not define $T\M$ itself here.)

\begin{rmk}
The tangent complex \eqref{eq:tangent} admits a natural geometric interpretation. In classical differential geometry,
if $Q$ is a vector field on a  manifold $\M$ and $P$ is
a point in $\M$ at which $Q$ vanishes, the Lie derivative $\LL_Q$ induces
a well-defined linear  map
$$\LL_Q: T\M |_P \to T\M |_P$$
called the linearization of $Q$ at $P$ \cite[p.72]{MR515141}. 

For an $L_\infty$-bundle $\M=(M,L,\lambda)$, let $Q_\lambda$ be the
 homological vector field as in (\ref{eq3}). The  classical
points  can be thought as those points  where  $Q_\lambda$
vanishes. Note that, if  $P$ is a classical point,
$TL|_{0_P} \cong T M|_P \oplus L|_P$, where $0_P\in L|_P$
is the zero vector. By a formal calculation,
one can easily check that the linearization of $Q_\lambda$ at $P$
gives rise to  exactly the tangent complex \eqref{eq:tangent} at $P$.
\end{rmk}

Let $\M=(M,L,\lambda)$ be an $L_\infty$-bundle. 
The \textit{virtual dimension} $\ddim(\M)$ of $\M$ is defined to be 
$$
\ddim(\M) = \dim(M) +\sum_{i=1}^n (-1)^i \rk(L^i). 
$$
Since $L$ is finite-dimensional, the Euler characteristic of the tangent complex $T\M|_P$ of $\M$ at each classical point $P$ is equal to the virtual dimension of $\M$.

Given a morphism of $L_\infty$-bundles $(f,\phi):\M\to\M'$,  and a classical point $P$ of $\M$, we get an induced morphism of tangent complexes
\begin{equation}\label{eq:TangentMap}
Tf|_P:T\M|_P\longrightarrow T\M'|_{f(P)}\,.
\end{equation}
Only the linear part $\phi_1$ of $\phi$ is used in $Tf|_P$. Below, for simplicity, we will denote a morphism $(f,\phi)$ by $f$.

\begin{defn}
Let  $f:\M\to\M'$ be a morphism of  $L_\infty$-bundles  and $P$ a classical
point of $\M$. We call $f$ {\it \'etale at $P$}, if the induced
morphism \eqref{eq:TangentMap} of tangent complexes at $P$ is a quasi-isomorphism (of complexes of vector spaces). 
The morphism $f$ is {\it \'etale}, if it is \'etale  at every
classical point of $\M$. 
\end{defn}

It is clear that every isomorphism of $L_\infty$-bundles is \'etale.

\begin{rmk}
For classical smooth  manifolds, an \'etale morphism is
exactly \eetale. On  the other hand, if both  the base manifolds
of $\M$  and $\M'$ are  just a point $\{*\}$, a morphism $f:\M\to\M'$
is \'etale means that $f$ is a quasi-isomorphism of
the corresponding $L_\infty[1]$-algebras if the 
curvature of $\M$ vanishes. Otherwise, it is always \'etale  since the set of classical points is empty in this case. 
\end{rmk}

\begin{defn}
A morphism $f:\M\to\M'$ of $L_\infty$-bundles is called a {\it weak
  equivalence} if
\begin{items}
\item $f$ induces a bijection on classical loci,
\item $f$ is \'etale.
\end{items}
\end{defn}

It is simple to check that weak equivalences satisfy two out of three, i.e. Axiom~\ref{axiom:2out3} in Definition~\ref{def:CFO} is indeed satisfied. Hence the category of $L_\infty$-bundles is a category with weak equivalences.

\begin{rmk}
If $\M \to \M'$ is a weak equivalence, then $\M$ and $\M'$ have the same virtual dimension. 
\end{rmk}

\subsection{Fibrations}

\begin{defn}\label{defnfib}
We call the morphism $(f,\phi):(M,L,\lambda)\to (M',L',\lambda')$ of
$L_\infty$-bundles a {\it fibration}, if 
\begin{items}
\item $f:M\to M'$ is a submersion,
\item $\phi_1:L\to f^\ast L'$ is a degree-wise surjective morphism of
  graded vector bundles over $M$.
\end{items}
\end{defn}

\begin{rmk}
If we consider $\phi_1$ as a bundle map of graded vector bundles
$$
\xymatrix{
L \ar[d] \ar[r]^{\phi_1} & L' \ar[d] \\
M \ar[r]_{f} & M',
}
$$
 the combination of  conditions (i)-(ii)
in Definition \ref{defnfib}
is  simply equivalent to that $\phi_1:L \to L'$ is a submersion
as a differentiable map.
\end{rmk}

For every $L_\infty$-bundle $\M$, the unique morphism $\M\to\ast$ is a (linear)
fibration. That is, Axiom~\ref{axiom:FibrantObj} in Definition~\ref{def:CFO} is indeed satisfied. It is also straightforward to show that Axiom~\ref{axiom:Fibration} in Definition~\ref{def:CFO} is true for the category of $L_\infty$-bundles.

If a fibrations is a linear morphism (Definition~\ref{defnlin}), we call it a {\em linear fibration}.

\begin{lem}\label{lem:IsoToLinearMor}
Every fibration is equal to 
the composition of a linear fibration with an isomorphism. 
\end{lem}
\begin{pf}
Let $(f,\phi):(M,L,\lambda)\to (N,E,\mu)$ be a fibration of $L_\infty$-bundles. By splitting the surjection $\phi_1:L\to f^\ast E$ we may assume that $L=f^\ast E\oplus F$, and that $\phi_1:L\to f^\ast E$ is the projection, which we shall denote by $\pi:L\to f^\ast E$.  We also denote the inclusion by $\iota:f^\ast E\to L$. 

We define operations $\phi_k':\Sym^k L\to L$ of degree zero by $\phi_1'=\id_L$, and $\phi_k'=\iota\,\phi_k$, for $k>1$. 

There exists a unique sequence of operations $\lambda':\Sym^k L\to L$ of degree 1, such that $\phi'\circ\lambda=\lambda'\bullet\phi'$. (Simply solve this equation recursively for $\lambda_k'$.) Thus
$ (M,L,\lambda')$ is an $L_\infty$-bundle, and
 $(\id_M,\phi'):(M,L,\lambda)\to (M,L,\lambda')$ is an isomorphism of $L_\infty$-bundles (see Proposition~\ref{oneiso}).

Then we have $\pi\bullet\phi'=\phi$ and $\mu\bullet\pi=\pi\circ \lambda'$. The latter equation can be checked after applying $\bullet \phi'$ on the right, as $\phi'$ is an isomorphism. Note that we   have $\phi\circ\lambda=\mu\bullet\phi$, by assumption, and hence $(\mu\bullet \pi)\bullet\phi'=\mu\bullet(\pi\bullet\phi')=\mu\bullet\phi=\phi\circ\lambda= (\pi\bullet\phi')\circ\lambda= \pi(\phi'\circ\lambda)=\pi(\lambda'\bullet\phi')=(\pi\circ\lambda')\bullet\phi'$.

We have proved that $(f,\pi):(M,L,\lambda')\to (N,E,\mu)$ is a linear morphism of $L_\infty$-bundles, and that $(f,\pi)\circ(\id_M,\phi')=(f,\phi)$. 
\end{pf}

\begin{rmk}
Suppose $\pi:L\to E$ is a linear fibration of curved $L_\infty[1]$-algebras
 over a manifold $M$.  Then $K=\ker\pi$ is a curved $L_\infty[1]$-ideal in $L$. 
This means that, for all $n$, if for some $i=1,\ldots,n$ we have $x_i\in K$,
 then $\lambda_n(x_1,\ldots,x_n)\in K$. 
(It does not imply that $\lambda_0\in K$.)
 Here $\lambda$ is the curved $L_\infty[1]$-structure on $L$. 
\end{rmk}

\begin{prop}\label{pullex} 
\begin{items}
\item Pullbacks of  fibrations exist in the category of
$L_\infty$-bundles, and  are still fibrations.
\item Pullbacks of $L_\infty$-bundle fibrations induce pullbacks of classical loci. 
\end{items}
\end{prop}
\begin{pf}
Let $(p,\pi):(M,L,\lambda)\to (N,E,\mu)$ be a fibration,
 and $(f,\phi):(M',L',\lambda') \to (N,E,\mu)$ an arbitrary morphism of $L_\infty$-bundles. According to Lemma~\ref{lem:IsoToLinearMor}, we may assume that $\pi= \pi_1$ is linear without loss of generality. We denote the corresponding sheaves of algebras by 
$\A=\Sym_{\O_M} L^\vee$, $\A'=\Sym_{\O_{M'}} {L'}^\vee$, and $\B=\Sym_{\O_N} E^\vee$. 

We define the manifold $N'$ to be the fibered product $N'= M\times_N M'$, and endow it with the sheaf of differential graded algebras 
$$\B'=\A|_{N'}\otimes_{\B|_{N'}}\A'|_{N'}\,.$$
Here $\A|_{N'}$, $\A'|_{N'}$ and $\B|_{N'}$ denote the pullback sheaves of $\A$, $\A'$ and $\B$ over $N'$, respectively. And the tensor product is regarded as a tensor product of sheaves of algebras over $N'$.  
To prove $(N',\B')$ is an $L_\infty$-bundle, let $E'=L \times_E L'$ be the fibered product of the smooth maps $\pi_1:L \to E$ and $\phi_1:L' \to E$, and
consider the  projection map $E'\to N'$. It is
straightforward to see that the latter is naturally  a graded vector bundle, and by the linearity of $\pi$,  
$$
\B' \cong \big((\Sym_{\O_{M'}}K^\vee)|_{N'} \otimes_{\O_{N'}}  \B|_{N'} \big)  \otimes_{\B|_{N'}}\A'|_{N'} \cong \Sym_{\O_{N'}}{E'}^\vee,
$$  
where $K=\ker(\pi_1)$ is a subbundle of $L$ such that $L \cong  K \oplus (p^\ast E) $. It is clear that $(N', E')\to (M', L')$
is a fibration.  
It follows from the fact that tensor product is a coproduct in the category of sheaves of differential graded algebras, that $(N',\B')$ is the fibered product of $(M,\A)$ and $(M',\A')$ over $(N,\B)$ in the category of differential graded manifolds. 

For the second claim, about classical loci, one can prove it by Remark~\ref{finalremark} and the universal property of pullbacks.
\end{pf}

Thus, Axiom~\ref{axiom:PullbackFib} in Definition~\ref{def:CFO} is established.

\begin{rmk}\label{rmk:DerDim&FibProd}
Let  $\M \to \N$ be a fibration, and $\M'\to \N$ an arbitrary morphism of $L_\infty$-bundles. If the underlying graded vector bundles of $\M$, $\M'$ and $\N$ are $(M,L)$, $(M',L')$ and $(N,E)$, respectively, then the underlying graded vector bundle of $\M \times_{\N} \M'$ is $(M \times_N M', L\times_E L')$. Thus, the 
virtual dimensions satisfy the   relation
$$
\ddim(\M \times_{\N} \M') = \ddim(\M) + \ddim(\M') - \ddim(\N)\,.
$$
\end{rmk}

\begin{prop}\label{prop:PullbackEtale}
Pullbacks of \'etale fibrations are \'etale fibrations.
\end{prop}
\begin{pf}
Let $\M\to \N$ be an \'etale fibration of $L_\infty$-bundles,
 and $\M'\to \N$ an arbitrary morphism of $L_\infty$-bundles.
 Let $\N'=\M\times_{\N}\M'$ be the fibered product, and $S'$  a classical point of $\N'$. 
 Denote the images of $S'$ under the maps $N' \to N$, $N' \to M$ and $N' \to M'$ by $S$, $P$, and $P'$, respectively. One checks that
$$T\N'|_{S'}=T\M|_P\times_{T\N|_S}T\M'|_{P'}\,.$$
The result follows from the fact that the pullback of a surjective quasi-isomorphism of cochain complexes along any cochain map is still a surjective quasi-isomorphism. 
See Example~\ref{ex:CFOcomplex}.
\end{pf}

By a similar method, one shows that the pullback of a weak equivalence along an $L_\infty$-bundle fibration is a weak equivalence.

Recall that a trivial fibration is a weak equivalence which is also a fibration. The following corollary is a consequence of Proposition~\ref{pullex} and Proposition~\ref{prop:PullbackEtale}.  It implies that Axiom~\ref{axiom:PullbackTriFib} in Definition~\ref{def:CFO} is indeed satisfied.

\begin{cor}
Pullbacks of trivial fibrations are trivial fibrations.
\end{cor}

Therefore, the only thing left to prove that $L_\infty$-bundles form a category of fibrant objects is the factorization property.

Following is another property of fibrations.  

\begin{prop}\label{prop:OpennessEtale}
Let $(f,\phi):(M,L,\lambda)\to (N,E,\mu)$ be a fibration of $L_\infty$-bundles. Let $P$ be a classical point of $(M,L,\lambda)$. Suppose that $(f,\phi)$ is  \'etale at $P$. Then there exists an open neighborhood $U$ of $P$ in $M$, such that $(f,\phi)$ is   \'etale at every classical point of $(U,L|_U,\lambda|_U)$. 
\end{prop}
\begin{pf}
For any $L_\infty$-bundle $\M= (M,L,\lambda)$, the tangent complexes at the points of $Z=\pi_0(\M)$ fit together into a topological vector bundle over the topological space $Z$. The morphism $(f,\phi)$ defines a degree-wise surjective morphism of complexes over $Z$. Hence the kernel is a complex of vector bundles on $Z$, which is acyclic at the point $P\in Z$.  For a bounded complex of topological vector bundles, the acyclicity locus is open.
\end{pf}

\subsection{The shifted tangent $L_\infty$-bundles}\label{sec:Ping}

\newcommand{\qQ}{Q}
\newcommand{\tQ}{\hat{\qQ}[-1]}
\newcommand{\tQQ}{\tilde{\qQ}}
\newcommand{\calx}{\XX}
\newcommand{\Aaa}{\AA}
\newcommand{\plambda}{\hat{\lambda}}

Let $\M=(M,\A, \qQ)$ be a dg manifold. Let $\Omega^1_\A$ be the sheaf of
smooth one forms on $\M$, which is a sheaf of graded $\A$-modules.
Then $\Sym_\A(\Omega^1_\A[1])$ is  a sheaf of  graded
$\O_M$-algebras on $M$, and
 $\big(M, \Sym_\A(\Omega^1_\A[1])\big)$ is a graded manifold, denoted  $T\M[-1]$. 
The Lie derivative $\LL_\qQ$ with respect to $\qQ$ defines the structure of a
sheaf of differential graded $\A$-modules on $\Omega^1_\A$. We pass to
$\Sym_\A(\Omega^1_\A[1])$ to define  a sheaf of differential graded
$\rr$-algebras on $M$. Thus  $T\M[-1]=(M, \Sym_\A(\Omega^1_\A[1]), \LL_\qQ)$
is  a dg manifold.

\begin{rmk}
For  a  graded manifold $\M=(M,\A)$, $T\M$ is the graded manifold
$(M, T_\A)$, where $T_\A=\Sym_\A \Omega^1_\A$ is a sheaf of  graded
$\O_M$-algebras on $M$. 
The bundle $T\M \to \M$ is called the tangent bundle of $\M$, and
it is a vector bundle in the category of graded manifolds.
If $\M=(M,\A, \qQ)$ is a dg manifold, 
it is standard \cite{MR3319134,MR3754617} that its tangent bundle $T\M$ is naturally  equipped
with a homological vector field $\hat{\qQ}$, called the complete lift \cite{MR0350650}, which makes it into a dg manifold.
The degree $1$ derivation $\hat{\qQ}: T_\A\to T_\A$ is
essentially induced by the Lie   derivative $\LL_\qQ$. 
 According to Mehta \cite{Mehta},
$\hat{\qQ}$ is a linear vector field with respect to
the tangent vector bundle $T\M\to \M$, therefore the 
shifted functor makes sense \cite[Proposition 4.11]{Mehta}. 
As a consequence, $T\M[-1]$, together with the homological
vector field $\tQ$,   is a dg manifold. It is simple
to check that the resulting dg manifold 
coincides with  the dg manifold  $(M, \Sym_\A(\Omega^1_\A[1]), \LL_\qQ)$ 
described above. Note that in this case
both $T\M\to \M$ and $T\M[-1]\to \M$  are vector bundles in the category
of dg manifolds.
\end{rmk}

If $(M,\A,Q)$ comes from an $L_\infty$-bundle $(M,L,\lambda)$ via the comparison functor (\ref{eq3}), 
then so does $(M,\Sym_\A(\Omega^1_\A[1]),\LL_\qQ)$ as we shall  see  below.

The underlying graded manifold of $T\M[-1]$ is $TL[-1]$, which
admits a double vector bundle structure
\begin{equation}\label{squr1}
\vcenter{\xymatrix{TL[-1]\rto\dto &TM[-1]\dto\\
L\rto &M }}
\end{equation}
A priori, $TL[-1]$ is {\it  not} a vector bundle over $M$.
However, by choosing a linear connection  $\cntnL$ on $L\to M$,
 one  can identify  $TL$ with $TM \times_M L \times_M L$.
Hence, one obtains a diffeomorphism   
\begin{equation}
\label{eq:phinabla}
\phi^{\cntnL}: \ \ TL[-1]\xxto{\cong}  TM[-1] \times_M L[-1] \times_M  L \, .
\end{equation} 
The latter is a graded vector bundle over $M$, namely it is $ TM[-1] \oplus L[-1] \oplus L$. 
On the level of sheaves, the diffeomorphism \eqref{eq:phinabla} is equivalent to
a splitting of the following
 short exact sequence of sheaves of graded $\A$-modules over $\O_M$
\begin{equation}\label{notsplit}
\xymatrix{
0\rto & \Omega^1_M [1]\otimes_{\O_M}\A\rto & \Omega^1_\A [1]\rto &
\Omega^1_{\A/\O_M}[1]\rto & 0\rlap{\,.}}
\end{equation}
where  $\A=\Sym_{\O_M}L^\vee$.

Thus one can transfer, via $\phi^{\cntnL}$, the homological vector field
  $\LL_\qQ$ on $TL[-1]$ into a  homological vector field $\tQQ$ on
the graded vector bundle $TM[-1]\oplus L[-1]\oplus L$. In what follows, we show that
$\tQQ$ is indeed tangent to the fibers of the graded vector  
bundle $TM[-1]\oplus L[-1]\oplus L$. Therefore, we obtain
an $L_\infty$-bundle with base manifold $M$.
To describe the $L_\infty[1]$-operations on
 $TM[-1]\oplus L[-1]\oplus L$, induced by $\tQQ$,  
let us introduce a formal variable $dt$ of degree $1$, and write
 $TM[-1]=TM\,dt$ and $L[-1]=L\, dt$, respectively. In the sequel, we will use
both notations $[-1]$ and $dt$ interchangely.

\begin{prop}
\label{pro:TM1}
Let $\M=(M,L,\lambda)$ be an $L_\infty$-bundle.
Any linear connection $\cntnL$ on $L$ induces an $L_\infty$-bundle
structure on $(M, \, TM \, dt\oplus L \, dt \oplus L,\,  \mu)$, where 
$$\mu  = \lambda+\tilde\lambda+\cntnL\lambda\,.$$ 
Here, for all $k\geq0$, $\lambda$, $\tilde\lambda$, and $\cntnL\lambda$ are
given, respectively, by  
\begin{eqnarray}
&&\lambda_k(\xi_1 ,\ldots,\xi_k)=
\lambda_k(x_1,\ldots,x_k), \label{eq:lambda}\\
&&\tilde\lambda_k(\xi_1 ,\ldots,\xi_k)=
\sum_{i=1}^k(-1)^{|x_{i+1}|+\ldots+ |x_k|}\lambda_k(x_1,\ldots,y_i,\ldots,x_k)\,dt\, , \\
&&(\cntnL\lambda)_{k+1}(\xi_0 ,\ldots,\xi_k)
=\sum_{i=0}^k(-1)^{|x_{i+1}|+\ldots+ |x_k|}(\cntnL_{v_i}\lambda_k)(x_0,\ldots,\hat
x_i,\ldots,x_k)\,dt\, , \qquad
\end{eqnarray}
$\forall \, \xi_i = v_i \, dt + y_i \, dt + x_i \in TM \, dt\oplus L \, dt \oplus L$,   
$i=0, \ldots k$.

Moreover,  different choices of linear connections $\cntnL$ on $L$ induce
isomorphic $L_\infty$-bundle structures on $TM \, dt\oplus L\, dt\oplus L$.
\end{prop}
\begin{pf}
Let  $\A=\Sym_{\O_M}L^\vee$ be the sheaf of functions on $\M$,
$R=\Gamma (M, \O_M)$  the algebra of $C^\infty$-functions on $M$,
  and
$\Aaa=\Gamma (M, \Sym L^\vee)$ the space of the global sections
of $\A$  being considered as  fiberwise  polynomial functions on $L$.

To describe the homological vector field $\tQQ$ on the graded vector
 bundle $TM \, dt \oplus L \, dt \oplus L$, 
we  will  compute its induced  degree 1 derivation on the
algebra of global sections of the  sheaf of functions.
First of all,  since  $Q$ is tangent to the fibers of 
 $\pi: L\to M$, it follows
that 
$$ \tQQ (f)= L_Q (\pi^* f)=0,  \ \ \forall f \in  R.$$
 As a consequence, the induced homological
vector field $\tQQ$ on  $TM\, dt\oplus L\, dt \oplus L$  is tangent
to the fibers. Hence we indeed obtain an $L_\infty$-bundle
on $TM\, dt\oplus L\, dt\oplus L$ with base manifold $M$.

To compute the $L_\infty[1]$-operations, we need to compute $\tQQ$
on the three types of generating functions $\Omega^1 (M)[1]$, $ \Gamma (M,  L^\vee[1])$ and
 $\Gamma (M, L^\vee)$, by using the identification  $\phi^{\cntnL}$ in  \eqref{eq:phinabla}
and    applying $\LL_Q$.

Note that the identification  $\phi^{\cntnL}$ induces the identity map
on $\Aaa=\Gamma (M, \Sym L^\vee)$. Therefore,
 for any  $\xi \in \Gamma (M, L^\vee)$,
$$\tQQ \xi=\LL_Q \xi=Q (\xi)=\sum_{k=0}\lambda_k^\vee (\xi )\in \bigoplus_k 
 \Gamma (M, \Sym^k L^\vee)$$
Hence, by taking its dual,   we obtain $\lambda$ as in Eq. \eqref{eq:lambda}.

The map  $\phi^{\cntnL}$ in  \eqref{eq:phinabla} induces an isomorphism:
\begin{equation}
\label{eq:psi}
\psi^{\cntnL}: \ \Gamma(M,\Omega_\A^1[1]) \xxto{\cong}   \big ( \Omega^1 (M) [1]\oplus \Gamma (M,  L^\vee[1])\big)\otimes_{R}\Aaa \,.
\end{equation}

From now on, we identify the  right hand side with the left hand side.
Denote by $d$ the composition of the de Rham differential 
(shifted by degree $1$) with the isomorphism $\psi^{\cntnL}$ as in
\eqref{eq:psi}: 
$$d: \Aaa\xxto{d_{DR}} \Gamma(M,\Omega_\A^1[1]) \xxto{\psi^{\cntnL}}   \big ( \Omega^1  (M) [1]\oplus \Gamma (M,  L^\vee[1])\big)\otimes_{R}\Aaa$$
By $d^{\cntnL}$,  we denote the covariant differential (shifted by degree $1$)
induced by the linear connection $\cntnL$:
$$d^{\cntnL}: \ \   \Aaa \to \Omega^1 (M)[1] \otimes_{R}\Aaa \hookrightarrow
 \big ( \Omega^1 (M) [1]\oplus \Gamma (M,  L^\vee[1])\big)\otimes_{R}\Aaa$$
It is simple to see that the image
 $(d-d^{\cntnL})(\Aaa) \subseteq  \Gamma (M,  L^\vee[1])\otimes_{R}\Aaa$.
Thus we obtain a map 
$$d_\rel:  \Aaa\to \Gamma (M,  L^\vee[1])\otimes_{R}\Aaa$$
such that 
$$d_\rel=d-d^{\cntnL} .$$
It is easy to check that  $d_\rel$ is in fact  $R$-linear, and therefore
it corresponds to a bundle map 
$$d_\rel: \ \Sym L^\vee \to L^\vee[1]\otimes \Sym L^\vee$$
over $M$, by abuse of notations.
 Furthermore, for any $\eta \in  \Gamma (M, L^\vee)$, 
we have 
$$d_\rel \eta =\eta [1] \, .$$
That is, when being restricted to  $L^\vee$, the bundle
map $d_\rel$ becomes the degree shifting
map  $L^\vee \to L^\vee [1]$,
and for general  $\Sym L^\vee$, one needs to apply the Leibniz rule.

Since $[\LL_Q, \ d]=0$,  it thus  follows that, for any
$\eta \in \Gamma (M, L^\vee)$,
$$ \tQQ (\eta[1])=  \LL_Q\circ d_\rel \eta= - d_\rel\circ \LL_Q\eta - [\LL_Q, d^{\cntnL}]\eta\,.$$ 
Now $ \LL_Q\eta\in \Aaa$ and $-d_\rel\circ \LL_Q\eta\in \Gamma (M, L^\vee[1])\otimes_{R}\Aaa$.
 By taking its dual, it gives rise to $\tilde\lambda$.

On the other hand, it is simple to see that 
$$[\LL_Q, d^{\cntnL}](\eta ) =(d^{\cntnL} Q)(\eta )\in \Gamma (M, T^\vee M \otimes \Sym L^\vee), $$
where the homological vector field $Q$ is considered as a section
in $\Gamma (M, 	\Sym L^\vee \otimes L)$, and
 $d^{\cntnL} Q\in \Gamma (M, T^\vee M \otimes \Sym L^\vee \otimes L)$. Hence it follows
that $-[\LL_Q, d^{\cntnL}](\eta)=-(d^{\cntnL} Q)(\eta)\in \Gamma (M, T^\vee M \otimes \Sym L^\vee)$.
 By taking its dual, we obtain  ${\cntnL}\lambda$.

Finally,  since $\pi_* Q=0$, it follows that
$$  \tQQ \theta=\LL_Q(\pi^* \theta )=0, \ \ \forall \theta\in \Omega^1 (M) [1]. $$
Hence the $\Omega^1 (M) [1]$-part does not contribute to  any  $L_\infty[1]$-operations. Therefore we conclude that $\mu = \lambda + \tilde \lambda + \cntnL \lambda$.

To prove the last part, let $\cntnL$ and $\bar\nabla$ be any two linear connections on $L$. 
Denote by $\mu^{\cntnL}$ and $\mu^{\bar\nabla}$ their corresponding $L_\infty[1]$-operations on $TM \, dt \oplus L\, dt \oplus L$, respectively. 
There exists a bundle map $\alpha:TM \otimes L \to L$ such that 
$
\alpha(v,x) = \bar\nabla_v x - \cntnL_v x, \, \forall \, v \in \Gamma(TM), \, x \in \Gamma(L).
$ 
Let $\phi^{\bar\nabla}$ and $\phi^{\cntnL}$ be the maps defined as in \eqref{eq:phinabla}.  Although the map $ \phi^{\bar\nabla} \circ (\phi^{\cntnL})^{-1}$ is not an isomorphism of vector bundles over $M$, it induces an isomorphism of $L_\infty$-bundles:
\begin{equation}\label{eq:CanonicalIsoTM[-1]}
 \Phi^{\bar\nabla,\cntnL} :\big( M, TM \, dt \oplus L\, dt \oplus L, \mu^{\cntnL} \big) \to \big( M, TM \, dt \oplus L\, dt \oplus L, \mu^{\bar\nabla} \big)
\end{equation}
given explicitly by 
\begin{align*}
&\Phi^{\bar\nabla,\cntnL}_1 = \id, \\
&\Phi^{\bar\nabla,\cntnL}_2(\xi_1, \xi_2)= \alpha(v_2,x_1) \, dt + (-1)^{|x_2|}\alpha(v_1,x_2) \, dt,  \\
&\Phi^{\bar\nabla,\cntnL}_n  = 0, \quad \forall n \geq 3,
\end{align*}
for any $\xi_i = v_i\, dt + y_i \, dt + x_i \in TM \, dt \oplus L\, dt \oplus L$, $i =1,2$. 
This concludes the proof.
\end{pf}

\begin{defn}\label{def:ShiftedTangentBundle}
Let $\M=(M,L,\lambda)$ be an $L_\infty$-bundle.
 Choose a linear connection $\cntnL$ on $L$.
We call the induced  $L_\infty$-bundle $(M,TM\, dt\oplus L\, dt\oplus L, \mu)$
as in Proposition \ref{pro:TM1} the {\it shifted tangent bundle }of $\M$, denoted $T\M[-1]$.
\end{defn}

The shifted tangent bundle of an $L_\infty$-bundle corresponds to the shifted tangent bundle of the corresponding dg manifold of finite positive amplitude under the equivalence provided by Proposition~\ref{pro:Batchelor}.

\section{The derived path space}\label{sec:DerPathSp}

The purpose of this section is to construct the derived path space of an
$L_\infty$-bundle $\M$. This will lead to a proof of the factorization
property, i.e.\ Axiom~\ref{axiom:Factorization} in Definition~\ref{def:CFO}, for $L_\infty$-bundles.

\subsection{Short geodesic paths}\label{sec:ShortGeodesic}

To construct derived path spaces, we need short geodesic paths, by which we mean geodesic paths defined on a fixed open interval containing $[0,1]$. 

We choose an affine connection $\cntnM$ on $M$, which gives rise to the notion of geodesic path in $M$, and defines an exponential map for $M$.

Let $I=(a,b)$ be an open interval containing $[0,1]$.  We will use $I$ as the domain for all our paths.

Using the exponential map, one shows the following.

\begin{prop}\label{prop:PgM}
There exists a family of geodesic paths, parameterized by a manifold $P_gM$ 
\begin{equation}
\label{eq:geodesic}
P_gM\times I\longrightarrow M\,, (\gamma, t)\to \gamma (t) 
\end{equation}
such that, in the diagram
\begin{equation}
\label{eq:PgM}
\begin{split}
\xymatrix{
& P_gM\drto^{\gamma(0)\times \gamma(1)}\dlto_{\gamma(0)\times \gamma'(0)}\\
TM&& M\times M\\
&M\rlap{\,,}\ulto^0\uuto^{\text{\rm const}}\urto_\Delta}
\end{split}
\end{equation}
the two upper diagonal maps are open embeddings. We call $P_gM$ a manifold of {\em short geodesic paths }in $M$. 
\end{prop}

For any $t\in I$, the map $\ev_t: P_g M\to M$, $t\mapsto \gamma (t)$,
(see \eqref{eq:geodesic}) is called the {\it evaluation map}.

Note that as a  smooth manifold, $P_gM$ is diffeomorphic to an 
 open neighborhood of  the zero section of the tangent bundle $TM$
via the left  upper diagonal map in diagram \eqref{eq:PgM}.
Under such an identification,  the evaluation map
$\ev_0: P_gM\to M$ becomes the projection $ TM \to M$,  while
$\ev_1: P_gM\to M$ becoming the exponential map $TM \to M, \, v_P\to\exp^{\cntnM}_P (v_P)$.
The combination of upper diagonal maps in diagram \eqref{eq:PgM}
gives rise to a diffeomorphism from an open neighborhood of  the zero section of the tangent bundle $TM$ to an open neighborhood of the diagonal
of $M\times M$, known as tubular  neighborhood theorem of smooth manifolds.

Let $L$ be a vector bundle over $M$. Let $\cntnM$ be an affine connection on $M$, and $\cntnL$ a linear connection on $L$. We will need the following two vector bundles, $P_{\con}L$ and $P_{\lin}L$, over $P_gM$:

\begin{defn}\label{def:CovConstPath}
The \textit{bundle of covariant constant paths in $L$} is the vector bundle $P_{\con}L$ over $P_gM$ whose fiber over a short geodesic path $a \in P_gM$ is the vector space of covariant constant sections $\Gamma_{\con}(I,a^\ast L)$ of $a^\ast L$ over $I$. 
Here, a \textit{covariant constant section} $\alpha$ of $a^\ast L$ is a path $\alpha: I \to L$ over $a:I \to M$ whose induced path in $L|_{a(0)}$ (via the parallel transport with respect to $\cntnL$) is a constant path. 
\end{defn}

It is easy to see that $P_{\con}L$ is a smooth vector bundle over $P_gM$ since it can be identified with the pullback bundle of $L$ to $P_g M$ via the evaluation
map $\ev_0:P_gM\to M$ at $0$. 
In particular, by taking $L=TM$, we have the vector bundle $P_{\con}TM$ over $P_g M$.

\begin{defn}\label{def:CovLinPath}
The \textit{bundle of covariant linear paths in $L$} is the vector bundle $P_{\lin}L$ over $P_gM$ whose fiber over a short geodesic path $a \in P_gM$ is the vector space of covariant linear sections $\Gamma_{\lin}(I,a^\ast L)$ of $a^\ast L$ over $I$. Here, a \textit{covariant linear section} $\alpha$ of $a^\ast L$ is a path $\alpha: I \to L$ over $a:I \to M$ whose induced path in $L|_{a(0)}$ (via the parallel transport with respect to $\cntnL$) is a line. 
\end{defn}

As a smooth vector bundle over $P_gM$, the bundle $P_{\lin}L$ can be identified with $\ev_0^\ast L \oplus \ev_1^\ast L$.

\subsection{The derived path space of a manifold}\label{sec:coam}

Now we construct derived path spaces of a manifold $M$. Let $\cntnM$ be an affine connection on $M$, and $P_\con TM$ the induced bundle of covariant constant paths in $TM$. 
Notice that there is a canonical section of $P_\con TM$ over $P_gM$, which maps a path to its derivative.  We denote it
by $D:P_gM\to P_\con TM$. Then 
$$\P=(P_gM,P_\con TM\,dt,D)$$ is an $L_\infty$-bundle of amplitude~1, the
{\it derived path space }of $M$. (Here $dt$ is a formal variable of degree $+1$.) 
The classical locus of $\P$ is the
set of constant paths, as a geodesic path is constant if and only if
its derivative vanishes.

Mapping a point $P\in M$ to the constant path at $P$ gives rise to a
morphism of $L_\infty$-bundles $M\to \P$. In fact, this map is a weak equivalence: It is
a bijection on classical loci because the classical locus of $\P$
consists of the constant paths.  To see that it is \'etale, consider a
point $P\in M$, and the corresponding constant path $a:I\to
M$. Then we have a short exact sequence of vector spaces
\begin{equation}\label{ses1}
TM|_P\longrightarrow T(P_gM)|_a\longrightarrow TM|_P\,.
\end{equation}
The first map is the map induced by $M\to P_gM$ on tangent spaces, the
second is the derivative of $D$, occurring in the definition of the
tangent complex of $\P$. To see that this sequence is exact, identify
$P_gM$ with an open neighborhood of the zero section in $TM$. This
identifies $T(P_gM)|_a$ with $TM|_P\times TM|_P$. The first map in
(\ref{ses1}) is then the inclusion into the first component. The second map is the
projection onto the second component. Hence (\ref{ses1}) is a short exact sequence, and $M\to \P$ is \'etale.

Moreover, the product of the two evaluation maps defines a fibration
of $L_\infty$-bundles $\P\to M\times M$.  We have achieved the desired
factorization of the diagonal of $M$ via $\P$, proving that $\P$ is,
indeed, a path space object for $M$ in the sense of a category of fibrant objects. 

To sum up, we have the following

\begin{prop}
Let $M$ be a manifold. 
The triple $(P_gM,P_\con TM\,dt,D)$ is a path space object of $M$ in the category of $L_\infty$-bundles. 
\end{prop}

\begin{rmk}
The manifold $P_gM$ can be identified with an open neighborhood $U$ of the zero section in $TM$ via Proposition~\ref{prop:PgM}. Let $\pi:U \to M$ be the restriction of the projection $TM \to M$. Then the bundle $P_\con TM\,dt$ can be identified with $\pi^\ast TM[-1] \cong U \times_M TM[-1]$, and the section $D$ is identified with the tautological section $D:U \to \pi^\ast TM[-1], \, v \mapsto (v,v)$. 
With this identification, the diagonal map $\Delta: M \to M \times M$ factorizes as follows:
$$
\begin{tikzcd}[column sep=2.5cm]
M \times 0 \ar[d,two heads] \ar[r,"\text{fiberwise zero map}"] &  \pi^\ast TM[-1]\ar[d,two heads] \ar[r,"\text{fiberwise zero map}"] & (M \times M) \times 0 \ar[d,two heads] \\
M \ar[u,bend right,"0"'] \ar[r,"P \mapsto 0_P"'] & U \ar[u,bend right,"D"'] \ar[r,"v \mapsto  (\pi (v)  {,} \exp^{\cntnM}(v)) "'] & M \times M \ar[u,bend right,"0"']
\end{tikzcd}
$$
\end{rmk}

\subsection{The derived  path space of an $L_\infty$-bundle}\label{sec:DrivedPathSp}

Now we extend the construction of derived path space  to a general $L_\infty$-bundle
to prove the factorization theorem. 
The main theorem is summarized below. 

\begin{thm}
\label{thm:Seoul}
Let $\M=(M,L,\lambda)$ be an $L_\infty$-bundle. Choose an affine connection $\cntnM$
on $M$ and a linear connection $\cntnL$ on the vector bundle $L$. There is an induced
$L_\infty$-bundle 
\begin{equation}
\label{eq:Cov19}
\P\M=\big(P_g M, P_\con (TM\oplus L)\,dt\oplus P_\lin L, \, \delta + \nu \big)\,,
\end{equation}
such that
\begin{enumerate}
\item 
the bundle map
\begin{equation}
\label{eq:UBC1}
\begin{split}
\xymatrix{
L
 \ar[d] \ar[r]^-{\iota} & P_\con (TM\oplus L)\,dt\oplus P_\lin L \ar[d] \\
M \ar[r]_-{\iota} & P_g M
}
\end{split}
\end{equation}
is a {\em linear} weak equivalence of $L_\infty$-bundles  $\M \to \P\M$, where
$\iota$ stands for the natural inclusion map of constant paths: for $l \in L$, identified with an element in $P_{\lin}L$ --- the constant path with value $l$, 
$$
\iota(l) = 0 \cdot dt + l.
$$

\item The bundle map 
\begin{equation}
\label{eq:UBC2}
\begin{split}
\xymatrix@C=4pc{
P_\con (TM\oplus L)\,dt\oplus P_\lin L \ar[d] \ar[r]^-{(\ev_0\oplus \ev_1)\circ \pr} & \ev_0^*L \oplus \ev_1^*L \ar[d] \\
P_g M \ar[r]_-{\ev_0\times \ev_1} & M\times M
}
\end{split}
\end{equation}
is a {\em linear} fibration of $L_\infty$-bundles  $\P\M\to\M\times\M$. Here $\pr: P_\con (TM\oplus L)\,dt\oplus P_\lin L \to P_\lin L$ is the projection onto $P_\lin L$.

\end{enumerate}
In particular, the composition $
\M \to \P\M\to\M\times\M$ is a factorization of the diagonal morphism $\M \to \M \times \M$.

Moreover, for any two choices of affine connections on $M$ and linear connections on $L$, there exists an open neighborhood $U$ of the constant paths in $P_gM$ such that the restrictions to $U$ of their corresponding derived path spaces are isomorphic. 
\end{thm}

The $L_\infty$-bundle $\P\M$ in \eqref{eq:Cov19} is called the {\it derived path space }of $\M$.

\begin{rmk}\label{rmk:DerPathSpAsVB}
As  graded vector bundles  over $P_g M$,
 $P_\con (TM\oplus L)\,dt\oplus P_\lin L$ is 
 isomorphic to $\ev_0^*(TM\oplus  L)\,dt\oplus \ev_0^*L \oplus \ev_1^*L$,
 where $\ev_0, \ev_1 : P_g M \to M$ are the
evaluation maps at $0$ and $1$, respectively. (When
$P_g M$ is identified with an open neighborhood of the diagonal
of $M\times M$,  $\ev_0$ and  $\ev_1$ correspond
to the projections to the first and the second component, respectively.)
Under this identification, the bundle map \eqref{eq:UBC1} becomes

\begin{equation}
\label{eq:UBC3}
\begin{split}
\xymatrix{
L
 \ar[d] \ar[r]^-{0\oplus \Delta} & \ev_0^*(TM\oplus  L)\,dt\oplus \ev_0^*L \oplus \ev_1^*L \ar[d] \\
M \ar[r]_-{\iota} & P_g M
}
\end{split}
\end{equation}
 where $\Delta: L \to \ev_0^*L \oplus \ev_1^*L$ is the diagonal map.
Similarly,  the bundle map \eqref{eq:UBC2} becomes
\begin{equation}
\label{eq:UBC4}
\begin{split}
\xymatrix{
\ev_0^*(TM\oplus  L)\,dt\oplus \ev_0^*L \oplus \ev_1^*L \ar[d] \ar[r]^-{\pr} & \ev_0^*L \oplus \ev_1^*L \ar[d] \\
P_g M \ar[r]_-{\ev_0\times \ev_1} & M\times M.
}
\end{split}
\end{equation}
Here $\pr$ is the projection onto $ \ev_0^*L \oplus \ev_1^*L$.
\end{rmk}

\begin{rmk}
Theorem \ref{thm:Seoul} (2) is equivalent to stating
that for $t= 0$ and  $1$,  the bundle map
$$
\xymatrix{
P_\con (TM\oplus L)\,dt\oplus P_\lin L\ar[d] \ar[r]^-{\ev_t \circ \pr} & L  \ar[d] \\
P_g M \ar[r]_-{\ev_t} & M
}
$$
is a {\em linear} morphism of $L_\infty$-bundles  $\P\M\to\M$.
This, however, may be   false for arbitrary $t\in I$, which may
need a  morphism with  higher terms. 
\end{rmk}

Our approach is based  on path space construction of $L_\infty$-bundles  
together with the $L_\infty[1]$-transfer theorem.
We will  divide the proof of  Theorem \ref{thm:Seoul}
into several steps, which are discussed in subsequent subsections.

\subsection{The path space of the shifted tangent bundle}

To construct a derived path space $\P\M$, we first construct an infinite-dimensional curved $L_\infty[1]$-algebra over each path $a$ in $M$. The consideration of path spaces is motivated by AKSZ construction which is summarized in Appendix~\ref{sec:AKSZ}.

Recall from Proposition~\ref{pro:TM1} that, for each $L_\infty$-bundle $\M = (M,L,\lambda)$ and a linear connection $\cntnL$ on $L$, we have an associated $L_\infty$-bundle, the shifted tangent bundle,
 $T\M[-1] = (M,TM \, dt  \oplus L \, dt \oplus L, \mu)$. (See Definition~\ref{def:ShiftedTangentBundle}.) 
The family of $L_\infty[1]$-operations is
$\mu =\lambda+\tilde\lambda+\nabla\lambda$.

For every path $a:I\to M$ in $M$, we get an induced curved $L_\infty[1]$-structure in the vector space
$\Gamma(I,a^\ast T\M[-1])=\Gamma(I,a^\ast(TM \oplus L)\,dt\oplus a^\ast L)$, by pulling back $\mu$ via $a$.

The covariant derivative with respect to $a^\ast\nabla$ is  a linear map
$\delta:\Gamma(I,a^\ast L)\to \Gamma(I,a^\ast L)\,dt$.
That is, $\delta (l(t))=(-1)^{k} \, \nabla_{a' (t)}l(t)\, dt$, \, $\forall l(t)\in
\Gamma(I,a^\ast L^{k})$. 
As $\delta^2=0$, it induces the structure of a complex on
$\Gamma(I,a^\ast T\M[-1])$.

We also consider the derivative $a'\, dt \in \Gamma(I,a^\ast TM)\,dt$.

\begin{prop}
\label{pro:paris}
Let $\cntnL$ be a linear connection on $L$. 
The sum $a'\, dt \, +a^\ast\mu$ is a curved $L_\infty[1]$-structure
on the complex $\big(\Gamma(I,a^\ast T\M[-1]), \delta \big) = \big( \pathast ,    \delta\big)$.
\end{prop}

Here we prove Proposition~\ref{pro:paris} by  a  direct verification, 
without referring to the infinite dimensional dg manifold of path spaces.
 The discussion in Appendix~\ref{sec:AKSZ}, however, gives a heuristic
reasoning as to where such a curved $L_\infty[1]$-structure formula comes from.

\begin{pfLooPathSp}
The condition we need to check is
$$(\delta+a'\, dt +a^\ast\mu)\circ(\delta+a'\, dt +a^\ast\mu)=0\,.$$
This reduces  to the condition
$$\delta\circ a^\ast\lambda+a^\ast\tilde\lambda\circ\delta+ a^\ast\nabla\lambda\circ
a'\, dt=0\,,$$
as all other terms vanish, either for degree reasons, or because
$\mu\circ\mu=0$. 

Let $x_1,\ldots,x_n$, $n\geq0$, be homogeneous elements  in $\Gamma(I,a^\ast
L)$. We need to prove that
  \begin{multline*}
\delta\big((a^\ast\lambda_n)(x_1,\ldots,x_n)\big) 
+ \sum_i(-1)^{|x_i|(1+|x_1|+ \ldots + |x_{i-1}|)}(a^\ast\tilde \lambda_n)(\nabla_{a'} x_i\,dt,x_1,\ldots,\hat x_i,\ldots,
x_n) \\
+(a^\ast\nabla\lambda_{n})(a'\, dt ,x_1,\ldots,x_n)=0\,.
\end{multline*}
Using the definitions of $\tilde \lambda$, and $\nabla\lambda$, and
canceling the factor $dt$, we are reduced to
\begin{multline*}
(-1)^{1+|x_1|+ \ldots + |x_n|}\nabla_{a'} \big(\lambda_n(x_1,\ldots,x_n)\big) \\
+ \sum_i(-1)^{|x_i|(|x_1|+ \ldots + |x_{i-1}|) + |x_1|+ \ldots + |x_n|}(a^\ast \lambda_n)(\nabla_{a'} x_i,x_1,\ldots,\hat x_i,\ldots,
x_n) \\
+(-1)^{|x_1|+ \ldots + |x_n|}(\nabla_{a'} \lambda_{n} )(x_1,\ldots,x_n)=0\, ,
\end{multline*}
which is true by the definition of $\nabla_{a'} \lambda_{n}$. 
\end{pfLooPathSp}

The diagram below describes all operations with 0 or 1 inputs: 
$$\xymatrix@C=4pc{
a'\in \Gamma(I,a^\ast TM)\,dt\rto^-{a^\ast \nabla\lambda_0} &
\Gamma(I,a^\ast L^1)\,dt\rto^-{a^\ast\tilde\lambda_1}
&
\Gamma(I,a^\ast L^2)\,dt\rto^-{a^\ast\tilde\lambda_1}
&\ldots\\
a^\ast\lambda_0\in \Gamma(I,a^\ast L^1)\urto^-\delta\rto_-{a^\ast\lambda_1}  &
\Gamma(I,a^\ast L^2)\urto^-\delta\rto_-{a^\ast\lambda_1}& \Gamma(I,a^\ast L^3)\urto^-\delta\rto_-{a^\ast\lambda_1} &\ldots  }
$$

\begin{rmk}
When the base manifold  $M$ is a point $\{*\}$,  an $L_\infty$-bundle is simply a curved $L_\infty [1]$-algebra $({\mathfrak g}, \lambda)$.
One can check that the  curved $L_\infty [1]$-algebra described in
Proposition \ref{pro:paris}
 can be identified with the one on
 $\Omega (I)\otimes {\mathfrak g}$ induced from $({\mathfrak g}, \lambda)$
by tensoring the cdga $(\Omega (I), d_{DR})$, where   $d_{DR}$ stands
for the de Rham differential.
We are thus reduced to the case studied by Fiorenza-Manetti  \cite{FioMan}
(they consider $L_\infty$-algebras without curvatures and differential
forms of polynomial functions). 
\end{rmk}

\subsection{The construction of $\P\M$}

We now choose an affine connection $\cntnM$ on $M$, and restrict to the case where $a\in P_gM$ is a short geodesic path. 
Write $\Gamma_\con(I,a^\ast TM)$ for the subspace of covariant constant sections of $a^\ast TM$ over $I$. Then $D(a) = a'\, dt\, \in
\Gamma_\con(I,a^\ast TM)\,dt$, and we have a curved
$L_\infty[1]$-subalgebra 
$$\widetilde P T\M[-1]|_a := \Gamma_\con(I,a^\ast TM)\,dt\oplus\Gamma(I,a^\ast L)\,dt\oplus \Gamma(I,a^\ast L)$$ 
in $\Gamma(I,a^\ast TM)\,dt\oplus\Gamma(I,a^\ast L)\,dt\oplus \Gamma(I,a^\ast L)$.

Following the Fiorenza-Manetti method  \cite{FioMan}, we introduce the derived path space $\P\M$ by contraction and transfer. Define 
\begin{align*}
\eta: \Gamma(I,a^\ast L^k)\,dt&\longrightarrow \Gamma(I,a^\ast L^k)\\
\alpha(t)\,dt&\longmapsto(-1)^{k} \Big(\int_0^t\alpha(u)\,du-t\int_0^1\alpha(u)\,du\,\Big).
\end{align*}
The connection $\cntnL$
trivializes the vector bundle $a^\ast L$ over the one-dimensional
manifold $I$,
i.e.\ we have a canonical identification $a^\ast L\cong I\times V$, where 
$V=\Gamma_\con(I,a^\ast L)$. This identifies $\Gamma(I,a^\ast L)$ with $C^\infty(I,V)$. Then $\int \alpha(u)\,du$ is the integral of
the vector valued function $\alpha$ defined on the open interval $I$. 

We consider $\eta$ as a linear endomorphism of degree $-1$ of the graded vector space $\widetilde P T\M[-1]|_a$. 
 We check that
\begin{items}
\item $\eta^2=0$,
\item $\eta \delta \eta=\eta$. 
\end{items}
Therefore,  $\delta\eta$ and $\eta\delta$ are orthogonal idempotents in $\widetilde P T\M[-1]|_a$, and hence induce a decomposition: 
$$\widetilde P T\M[-1]|_a =H\oplus \im \delta\eta\oplus \im \eta\delta\,,\qquad H=\ker
\delta\eta\cap \ker \eta \delta\,.$$ 
The differential $\delta$ preserves $H$.  The  inclusion $\iota:H\to
\widetilde P T\M[-1]|_a$, and the projection 
$$\pi=1-[\delta,\eta]:\widetilde P T\M[-1]|_a \to H$$ 
set up a homotopy 
equivalence between $(H,\delta)$ and $(\widetilde P T\M[-1]|_a, \delta)$.

The transfer theorem for curved $L_\infty[1]$-algebras, Proposition~\ref{transfertheorem}, gives rise to a
curved $L_\infty[1]$-structure on $(H,\delta)$, and a morphism of curved
$L_\infty[1]$-algebras from $(H,\delta)$ to $(\widetilde P T\M[-1]|_a ,\delta)$.  To apply the transfer theorem, we use the filtration given by  $F_k=\Gamma(I,a^\ast L^{\geq k}\oplus L^{\geq k}\,dt)$, for $k\geq1$.

The kernel of $\delta\eta$ on $\Gamma(I,a^\ast L)\,dt$ is the space of
covariant constant sections  $\Gamma_\con(I,a^\ast L)\,dt$. The kernel of $\eta\delta$ on $\Gamma(I,a^\ast L)$ is the
space of {\em linear }sections, denoted $\Gamma_\lin(I,a^\ast
L)$. These are the maps $s:I\to V$ which interpolate linearly between $s(0)$
and $s(1)$, i.e.\ which satisfy the equation
$$s(t)=(1-t)\,s(0)+t\,s(1)\,,\qquad\text{for all $t\in I$\,.}$$

The projection $\pi_\lin:\Gamma(I,a^\ast L)\to \Gamma_\lin(I,a^\ast L)$ maps the path
$\alpha$ to the linear path with the same start and end point:
\begin{equation}\label{eq:piLin}
\pi_\lin(\alpha)(t)=(1-t)\alpha(0)+t\,\alpha(1)\,.
\end{equation}
The projection $\pi_\con:\Gamma(I,a^\ast L)\,dt\to
\Gamma_\con(I,a^\ast L)\, dt$ maps the
path $\alpha\, dt$ to its integral:
\begin{equation}\label{eq:piCon}
\pi_\con(\alpha\, dt)=\Big( \int_0^1\alpha( u )\,du\,\Big) \,  dt \, .
\end{equation}

Let $\P\M$ be a family of curved $L_\infty[1]$-algebras over $P_gM$, whose fiber over $a \in P_gM$ is given by the transfer theorem (Proposition~\ref{transfertheorem}): 
\begin{equation}\label{eq:FibDerPathSp}
\P\M|_a = \big(\Gamma_\con(I,a^\ast TM) \, dt  \oplus \Gamma_\con(I, a^\ast L)\,dt \oplus \Gamma_\lin(I,a^\ast L), \, \delta + \nu \big).
\end{equation}

\begin{prop}
The triple $(P_gM, \P\M, \delta + \nu)$ is an $L_\infty$-bundle.
\end{prop}
\begin{pf}
The vector spaces $\Gamma_\con(I,a^\ast TM)\,dt$, as $a$ varies over $P_gM$, combine into a vector bundle $P_\con TM$ over $P_gM$. Similarly, the vector spaces $\Gamma_\con(I,a^\ast L)\,dt$ and $\Gamma_\lin(I,a^\ast L)$ combine 
into vector bundles $P_\con L\,dt$ and $P_\lin L$ over $P_gM$, respectively. Note that, as vector bundles over $P_gM$, 
\begin{align}
P_\con L & \cong \ev_0^\ast L, \label{eq:PconL} \\
P_\lin L & \cong \ev_0^\ast L\oplus \ev_1^\ast L. \label{eq:PlinL}
\end{align}
See Definition~\ref{def:CovConstPath} and Definition~\ref{def:CovLinPath}.

We have constructed curved $L_\infty[1]$-structures in the fibers of $(P_\con TM\oplus P_\con L)\,dt\oplus P_\lin L$ over $P_gM$ by the transfer theorem.  These combine into a bundle of smooth
curved $L_\infty[1]$-algebras, because all ingredients,
 $D$, $\lambda$, $\delta$ and $\eta$ are differentiable with respect to the base manifold $P_gM$, and all transferred operations are given by explicit formulas involving finite sums over rooted trees (see Remark~\ref{rmk:TreeFormula}).   
\end{pf}

As a consequence of the transfer theorem, we have the following lemma.

\begin{lem}
For each $a \in P_gM$,there exist $L_\infty[1]$-morphisms
\begin{eqnarray} 
&&\phi: \P\M|_a \to   \widetilde P T\M[-1]|_a \qquad \label{eq:KIAS1}\\
&&\tilde{\pi}: \widetilde P T\M[-1]|_a \to \P\M|_a \qquad \label{eq:KIAS2}
\end{eqnarray}
satisfying the recursive equation \eqref{recu} and $\tilde\pi \bullet \phi = \id_{\P\M|_a}$.
\end{lem}

\subsection{Some properties of the $L_\infty[1]$-operations on $\P\M|_a$}

We compute some $L_\infty[1]$-operations in \eqref{eq:FibDerPathSp} which will be used in the future.

First, we apply Proposition~\ref{transfertheorem} to compute the zeroth and first
$L_\infty[1]$-operations in \eqref{eq:FibDerPathSp}. This computation is needed in the proof of Proposition~\ref{prop:WeakEqui}.

Recall that we transfer the $L_\infty[1]$-structure $\delta + a'\, dt \, + a^\ast\lambda +  a^\ast \tilde\lambda + a^\ast \cntnL\lambda$ on $\widetilde P T\M[-1]|_a$ to an $L_\infty[1]$-structure $\delta + \nu$ on $\Gamma_\con(I,a^\ast TM) \, dt  \oplus \Gamma_\con(I, a^\ast L)\,dt \oplus \Gamma_\lin(I,a^\ast L)$. The transfer is via the homotopy operator $\eta$, and the associated projection is $\pi = \pi_\con + \pi_\lin$ as defined in \eqref{eq:piLin} and \eqref{eq:piCon}. 
All the $L_\infty[1]$-operations and morphisms are considered as maps 
$\Sym E \to F$, where $E$ and $F$ denote the source and target spaces of the $L_\infty[1]$-operations/morphisms, respectively. 
The compositions are denoted by $\bullet$ or $\circ$ 
as defined in Section~\ref{linfty}.

The transferred curvature is 
\begin{equation}\label{eq:TransferredCurvature}
\pi(\mu_0)=a'\, dt \, +\pi_\lin (a^\ast \lambda_0)\,.
\end{equation}
Here the second term is the linearization of the curvature: the linear path
from the starting point of the curvature to the end point of the
curvature along the geodesic path $a$.

Let $\phi$ be the $L_\infty[1]$-morphism \eqref{eq:KIAS1}
 induced from the homotopy transfer, 
and 
$$\iota: 
\P\M|_a \to   \widetilde P T\M[-1]|_a$$  
the set-theoretical inclusion map. Since the homotopy operator
 $\eta$ vanishes on $\Gamma(I,a^\ast L)$ and $\Gamma(I,a^\ast TM )dt$, it follows
that $\eta(a'\, dt \, + a^\ast\lambda) \bullet \phi =0$. Hence the homotopy transfer equation \eqref{recu} reduces to 
\begin{equation}\label{eq:TransInc}
\phi   
= \iota - \eta \, (a^\ast \tilde\lambda) \bullet \phi - \eta\, (a^\ast \cntnL\lambda)\bullet \phi.
\end{equation}
The transferred $L_\infty[1]$-structure is $\delta+ \nu$, where 
\begin{equation}\label{eq:TransLoo}
\nu = \pi\, (a'\, dt \, + a^\ast \lambda + a^\ast\tilde\lambda + a^\ast\cntnL\lambda )\bullet \phi.
\end{equation}

\begin{lem}
\label{lem:lambda}
$
\pi(a^\ast \lambda)\bullet \phi 
= \pi(a^\ast \lambda)\bullet \iota.
$ 
\end{lem}
\begin{pf}
The output of $\eta$ consists of sections $\alpha\in\Gamma(I,a^\ast L)$
 such that $\alpha(0)=\alpha(1)=0$.
For any $n\geq 1$, by multi-linearity of $\lambda_n$,
  $(a^\ast \lambda_n)\big( \ldots, \eta(\text{sth}), \ldots\big)(t)=0$, for $t=0$ and
$t=1$.
  As the projection onto $P_\lin L$ is the linear interpolation, 
it follows that $\pi(a^\ast \lambda)\big( \ldots, \eta(\text{sth}), \ldots\big)=0$.
 Therefore, it follows from Equation~\eqref{eq:TransInc} and the definition of $\bullet$ that 
\begin{align*}
\pi(a^\ast \lambda)\bullet \phi & = \pi(a^\ast \lambda)\bullet \iota + \pi(a^\ast \lambda)\big( \ldots, \eta(\text{sth}), \ldots\big) \\
& = \pi(a^\ast \lambda)\bullet \iota. 
\end{align*}
\end{pf}

In particular, if we are given only one input, the recursive equation \eqref{eq:TransInc} reduces to 
\begin{equation}\label{eq:FirstMapInclusion}
\phi_1  = \iota - \eta \, (a^\ast \tilde\lambda_1 + a^\ast \cntnL\lambda_0) \phi_1.
\end{equation}
Therefore, we have 
\begin{align*}
\nu_1 & = \pi\, ( a^\ast \lambda_1) \phi_1 + \pi\, (a^\ast\tilde\lambda_1 + a^\ast\cntnL\lambda_0 ) \phi_1 \\
& = \pi\, ( a^\ast \lambda_1) \iota + \pi\, (a^\ast\tilde\lambda_1 + a^\ast\cntnL\lambda_0 ) \iota + \pi\, (a^\ast\tilde\lambda_1 + a^\ast\cntnL\lambda_0 ) \big(\eta(sth)\big) \\
& = \pi\, ( a^\ast \lambda_1) \iota + \pi\, (a^\ast\tilde\lambda_1 + a^\ast\cntnL\lambda_0 ) \iota \\
& = \pi_\lin \, a^\ast \lambda_1 + \pi_\con \, a^\ast\tilde\lambda_1 + \pi_\con \, a^\ast\cntnL\lambda_0. 
\end{align*}
Here we applied  Lemma \ref{lem:lambda} in  the second equality.

Consequently, we have the following 
\begin{prop}
Let $\delta + \nu$ be the $L_\infty[1]$-operations defined in \eqref{eq:FibDerPathSp}. 
The operations with $0$ and $1$ inputs are 
\begin{align*}
\nu_0 &= a'\, dt \, +\pi_\lin (a^\ast \lambda_0)\,, \\
\nu_1 & =\pi_\lin \, a^\ast \lambda_1 + \pi_\con \, a^\ast\tilde\lambda_1 + \pi_\con \, a^\ast\cntnL\lambda_0\, ,
\end{align*}
which are summarized in the following diagram:
\begin{equation}\label{eq:01inputs}
\begin{split}
\xymatrix@C=4pc{
a'\, dt \in  \Gamma_\con(I,a^\ast TM)\, dt\rto^-{\pi_\con a^\ast(\cntnL\lambda_0)} &
\Gamma_\con(I,a^\ast L^1)\,dt \rto^-{\pi_\con a^\ast\tilde\lambda_1 }
&\ldots\\
 \pi_\lin a^\ast\lambda_0\in
\Gamma_\lin(I,a^\ast L^1)\ar[ru]_-{\delta}\rto_-{\pi_\lin
  a^\ast\lambda_1} & 
\Gamma_\lin(I,a^\ast L^2)\ar[ru]_-{\delta}\rto_-{\pi_\lin a^\ast\lambda_1} & \ldots 
}
\end{split}
\end{equation}
\end{prop}

Next lemma is also needed in our future discussion.

\begin{lem}\label{lem:ForSubalg}
For any $v_1, \ldots, v_n \in \Gamma_\lin(I,a^\ast L)$, we have 
\begin{items}
\item\label{term:TransIncProp}
$\phi_1(v_1) = \iota(v_1)$, $\phi_n(v_1, \ldots, v_n) = 0, \; \forall \, n \geq 2$;
\item
$\nu_n(v_1, \ldots, v_n) = \pi_\lin \big((a^\ast\lambda_n)(v_1, \ldots, v_n)\big),  \; \forall \, n \geq 1$.
\end{items}
\end{lem}
\begin{pf}
By \eqref{eq:TransLoo}, the second statement follows from the first one and the fact that $a^\ast \tilde\lambda + a^\ast \cntnL\lambda$ vanishes on $\Gamma_\lin(I,a^\ast L)$. 
To prove the first assertion, we  will show the following claim by induction on $n$. 

{\it Claim.} The image of $\Gamma_\lin(I,a^\ast L)^{\otimes n}$ under $\phi_n$ is still in $\Gamma_\lin(I,a^\ast L)$. \\
In fact, since the operation $a^\ast \tilde\lambda + a^\ast \cntnL\lambda$ vanishes if all the inputs are elements in $\Gamma_\lin(I,a^\ast L)$, Equation \eqref{eq:TransInc}, together with the claim, imply that $\phi_n(v_1, \ldots, v_n) = 0$,  for any $n \geq 2$ and any $v_1, \ldots, v_n \in \Gamma_\lin(I,a^\ast L)$.

We start with computing $\phi_1$. Since the operator $\big( \eta (a^\ast \tilde\lambda_1 + a^\ast \cntnL\lambda_0)\big)^2$ vanishes, it follows from \eqref{eq:FirstMapInclusion} that  
\begin{align*}
\phi_1 & = \iota - \eta \, (a^\ast \tilde\lambda_1 + a^\ast \cntnL\lambda_0) \phi_1 \\
& = \iota - \eta \, (a^\ast \tilde\lambda_1 + a^\ast \cntnL\lambda_0) \big(\iota - \eta \, (a^\ast \tilde\lambda_1 + a^\ast \cntnL\lambda_0) \phi_1 \big) \\
& = \iota - \eta \, (a^\ast \tilde\lambda_1 + a^\ast \cntnL\lambda_0) \iota + \big( \eta (a^\ast \tilde\lambda_1 + a^\ast \cntnL\lambda_0)\big)^2\phi_1 \\
 & = \iota - \eta (a^\ast \tilde\lambda_1 + a^\ast \cntnL\lambda_0) \iota.
\end{align*}
Thus, the restriction of $\phi_1$ on $\Gamma_\lin(I,a^\ast L)$ coincides with $\iota$, which verifies the first part of  \ref{term:TransIncProp}.  Now assume the claim is true for $n$. Since $a^\ast \tilde\lambda + a^\ast \cntnL\lambda$ vanishes on $\Gamma_\lin(I,a^\ast L)$, it follows from \eqref{eq:TransInc} that the claim is also true for $n+1$.  
\end{pf}

\subsection{Proof of Theorem~A}

Let $\M=(M,L,\lambda)$ be an $L_\infty$-bundle.
Fix an affine connection $\cntnM$ on $M$ and a linear connection $\cntnL$ on $L$.  Consider the derived path space
$$\P\M=\big(P_g M, P_\con (TM\oplus  L)\,dt\oplus P_\lin L, \delta +\nu\big).$$
Let us denote the embedding of constant paths by $\iota:\M\to \P\M$.
By Lemma~\ref{lem:ForSubalg}, the map $\iota$ is an embedding of $L_\infty$-bundles.

\begin{prop}\label{prop:WeakEqui}
The morphism $\iota: \M\to \P\M$ of constant paths is a weak equivalence.
\end{prop}
\begin{pf}
The classical locus of $\M$ is the zero locus of $\lambda_0:M\to L^1$.  The classical locus of $\P\M$ is the zero locus of 
$$D+P_\lin\lambda_0:P_gM\longrightarrow P_\con TM \, dt \, \oplus P_\lin L^1: \, a \mapsto a'\, dt\, + \pi_\lin a^\ast \lambda_0 .$$
The canonical inclusion of the former into the latter is a bijection, because we can first compute the zero locus of $D$, which consists of constant paths. 

Let $P$ be a classical locus of $\M$.  
To compute the tangent complex of $\P\M$ at $P$, we identify $P_gM$ with a neighborhood of the diagonal in $M \times M$ by Proposition~\ref{prop:PgM}. The classical locus $P$ corresponds to $(P,P) \in M \times M$ under this identification. Furthermore, it follows from \eqref{eq:01inputs}, \eqref{eq:PconL} and \eqref{eq:PlinL} that the tangent complex of $\P\M$ at $P$ can be identified with the cochain complex: 
$$
\xymatrix@+1pc{TM|_P\oplus TM|_P\rto_-{D_P\nu_0} & TM\,dt|_P\oplus (L^1\oplus L^1)|_P\rto_-{(\delta+\nu_1)|_P } &L^1\,dt|_P\oplus(L^2\oplus L^2)|_P\rto_-{(\delta+\nu_1)|_P } & \ldots},
$$
where
\begin{align*}
& (D_P\nu_0)(v,w) = \big((w-v)\, dt, \, D_P\lambda_0(v), \, D_P\lambda_0(w)\big), \\
& (\delta+\nu_1)|_P(v\, dt, y_1,z_1) = \big(
 D_P\lambda_0(v) \, dt +  (y_1-z_1)\, dt\, ,\, \lambda_1(y_1), \,  \lambda_1(z_1) \big), \\
& (\delta+\nu_1)|_P(x_k \, dt, y_{k+1},z_{k+1}) = \big( \lambda_1(x_k) \, dt\, +(-1)^{k+1}(z_{k+1}-y_{k+1})\, dt\, , \, \lambda_1(y_{k+1}), \,  \lambda_1(z_{k+1}) \big),
\end{align*}
for $v,w \in TM|_P$, $x_k, y_k, z_k \in L^k|_P$. 
The first component $(w-v)\, dt$ in $(D_P\nu_0)(v,w)$  
is the composition 
$$
T(M\times M)|_{(P,P)} \xrightarrow{((\exp^{\cntnM})^{-1})_\ast} T(TM)|_{0_P} \xrightarrow{\pr} T^\fiber_{0_P}(TM) \cong T M|_P.
$$ 

Finally, the induced morphism on tangent complexes $T\M|_P \to T\P\M|_P$ at a classical point $P$ of $\M$ can be identified with the vertical cochain map:  
\begin{small}
\begin{equation}\label{eq:TangentMapWkEq}
\begin{split}
\xymatrix@+1pc{
TM|_P\rto^-{D_P\lambda_0}\dto_-{\Delta} & L^1|_P\rto^-{\lambda_1|_P}\dto_-{0\oplus\Delta} & L^2|_P\rto^-{\lambda_1|_P} \dto_-{0\oplus\Delta} &\ldots\\
TM|_P\oplus TM|_P\rto_-{D_P\nu_0} & TM\,dt|_P\oplus (L^1\oplus L^1)|_P\rto_-{(\delta+\nu_1)|_P } &L^1\,dt|_P\oplus(L^2\oplus L^2)|_P\rto_-{(\delta+\nu_1)|_P } & \ldots,}
\end{split}
\end{equation}
\end{small}
where $\Delta$ denotes the diagonal maps. 
Consider \eqref{eq:TangentMapWkEq} as a double complex. The vertical cochain map in \eqref{eq:TangentMapWkEq} is a quasi-isomorphism if and only if the total cohomology of the double complex \eqref{eq:TangentMapWkEq} vanishes. We compute the total cohomology by the spectral sequence associated with the filtration by columns. It follows from a direct computation that the first page is the mapping cone of the identity cochain map of the tangent complex of $\M$ at $P$. Thus, the second page vanishes, and the induced map on tangent complexes at $P$ is indeed a quasi-isomorphism, as desired.
\end{pf}

This concludes the proof of Theorem~\ref{thm:Seoul} (1).

We define a morphism $\P\M\to\M\times\M$ as the fiberwise composition of $L_\infty[1]$-morphisms 
\begin{equation}\label{eq:Evaluation}
P\M|_a \xrightarrow{\phi} \widetilde PT\M[-1]|_a 
\xrightarrow{\ev_0 \oplus \ev_1} L|_{a(0)} \oplus L|_{a(1)},
\end{equation}
where $\phi$ is as in \eqref{eq:KIAS1} which can be obtained by the homotopy transfer equation \eqref{eq:TransInc}, and $\ev_t$ is the composition of the projection onto $\Gamma(I, a^\ast L)$ followed by the evaluation map at $t$.

\begin{prop}\label{wow1}
The morphism $\P\M\to \M \times \M$ defined above 
is the linear morphism given by projecting onto $P_\lin L$ composing with evaluation at $0$ and $1$. 
\end{prop}
\begin{pf}
It suffices to show that the composition of $L_\infty[1]$-morphisms \eqref{eq:Evaluation} is a linear morphism. 
To prove linearity, we observe from \eqref{eq:TransInc} that $\im(\phi - \iota) \subset \im(\eta)$. Here $\phi$ and $\iota$ are considered as maps $\Sym\P\M|_a \to \widetilde PT\M[-1]|_a$. Thus, the linearity follows from the fact $\im(\eta) \subset \ker(\ev_0 \oplus \ev_1)$.
\end{pf}

This concludes the proof of   Theorem~\ref{thm:Seoul} (2). 

Now, we are ready to prove the main theorem.

\begin{thm}
\label{thm:main}
The category of $L_\infty$-bundles is a category of fibrant objects.
\end{thm}
\begin{pf}
It remains to show that every $L_\infty$-bundle $\M$ has a path space object. We claim that $\P\M$ is a path space object for $\M$. 
The composition $\M\to\P\M\to\M\times\M$ is clearly the diagonal morphism, and the morphism $ \M \to \P\M$ is a weak equivalence by Proposition~\ref{prop:WeakEqui}. 
Thus, it remains to show that the morphism $\P\M\to\M\times\M$ is a fibration of $L_\infty$-bundles. 

On the level of the base manifolds, this morphism is the open embedding $P_gM\to M\times M$. The linear part (which is the whole, by 
Proposition~\ref{wow1}) of the $L_\infty[1]$-morphism from
 $P_\con(TM\oplus L)\,dt \oplus P_\lin L$ to $\ev^*_0L\oplus \ev^*_1 L$ is simply the projection onto $P_\lin L$, followed by the identification $P_\lin L 
\cong \ev^*_0L\oplus \ev^*_1 L$.  This is clearly an epimorphism of vector bundles.
\end{pf}

\begin{rmk}
There is a very closely related model structure in
\cite[Proposition 3.12]{MR4036665}, to which we would
like to call readers' attention. 
\end{rmk}

\subsection{Canonicality of $\P\M$}

We now prove the last part of Theorem~\ref{thm:Seoul}: canonicality. First, different choices of affine connections on $M$ result diffeomorphic $P_gM$ (along an open neighborhood of constant paths). Hence, without loss of generality, we may fix any affine connection on $M$ and the corresponding $P_gM$.   

Let $\cntnL$ and $\bar\nabla$ be any two linear connections on $L$. We will construct a morphism $\Psi:\P^{\cntnL}\M \to \P^{\bar\nabla}\M$ of the corresponding $L_\infty$-bundles such that it is an isomorphism in an open neighborhood of constant paths in $P_gM$. 

Following the notations in the proof of Proposition~\ref{pro:TM1}, we denote by $\alpha$ the bundle map $\alpha:TM \otimes L \to L$ defined by $\alpha(v,x) = \bar\nabla_v x - \cntnL_v x, \, \forall \, v \in \Gamma(TM), \, x \in \Gamma(L)$, and denote by  $\Phi = \Phi^{\bar\nabla,\cntnL}$ the isomorphism \eqref{eq:CanonicalIsoTM[-1]}.

\begin{lem}
The morphism $\Phi$ induces a morphism of curved $L_\infty[1]$-algebras
$$a^\ast\Phi: \widetilde P^{\cntnL} T\M[-1]|_a \to \widetilde P^{\bar\nabla} T\M[-1]|_a$$ 
at each short geodesic path $a \in P_gM$. 
\end{lem}
\begin{pf}
To prove $a^\ast\Phi$ is a morphism, we check the defining equation of $L_\infty[1]$-morphism:
$$
a^\ast\Phi \circ (a'\, dt\, + \delta^{\cntnL} + a^\ast\mu^{\cntnL}) = (a'\, dt\, + \delta^{\bar\nabla} + a^\ast\mu^{\bar\nabla}) \bullet a^\ast\Phi
$$
which is considered as a sequence of equations labeled by the number of inputs. In fact, it is simple to see that the higher equations (with 2 or more inputs) follow from the fact that $\Phi$ is an $L_\infty[1]$-morphism: $\Phi\circ \mu^{\cntnL} = \mu^{\bar\nabla} \bullet \Phi$, and the zeroth equation $a^\ast\Phi_1(a' \, dt + a^\ast \mu^{\cntnL}_0) = a' \, dt + a^\ast \mu^{\bar\nabla}_0 $ simply follows from the definitions of $\mu^{\cntnL}$, $\mu^{\bar\nabla}$ and the fact that $\Phi_1 = \id$.  Thus, it remains to check the first equation:  
For any $\xi = v\, dt + y\, dt + x \in  \widetilde P^{\cntnL} T\M[-1]|_a$, we have 
\begin{align*}
& (a^\ast\Phi_2)(a' \, dt + a^\ast \mu^{\cntnL}_0 , \xi) + (a^\ast\Phi_1)\big((\delta^{\cntnL} + a^\ast\mu^{\cntnL}_1)(\xi) \big) \\
& \qquad = 
\alpha(v , a^\ast\lambda_0)\, dt\, + (-1)^{|x|} \alpha(a', x)\, dt \, \\
& \qquad \quad + (-1)^{|x|} (\cntnL_{a'} x) \, dt  \, + a^\ast(\cntnL_{v} \lambda_0) \, dt \, + a^\ast\lambda_1(y) \, dt \, + a^\ast\lambda_1(x) \\
& \qquad = (-1)^{|x|} (\bar\nabla_{a'} x) \, dt  \, + a^\ast(\bar\nabla_{v} \lambda_0) \, dt \, + a^\ast\lambda_1(y) \, dt \, + a^\ast\lambda_1(x) \\
& \qquad  =  (\delta^{\bar\nabla} + a^\ast \mu^{\bar\nabla}_1)\big(\Phi^{\bar\nabla,\cntnL}_1(\xi)\big).
\end{align*} 
This completes the proof. 
\end{pf}

By Proposition~\ref{transfertheorem}, we have the $L_\infty[1]$-morphisms
\begin{gather*}
\phi^{\cntnL}: \P^{\cntnL}\M|_a \longrightarrow \widetilde P^{\cntnL} T\M[-1]|_a, \\
\tilde\pi^{\bar\nabla}: \widetilde P^{\bar\nabla} T\M[-1]|_a\longrightarrow \P^{\bar\nabla}\M|_a.
\end{gather*}
See also \eqref{eq:KIAS1} and \eqref{eq:KIAS2}. We define $\Psi:\P^{\cntnL}\M \to \P^{\bar\nabla}\M$ by 
\begin{equation}\label{eq:Canonicality:Psi|_a}
\Psi|_a := \tilde\pi^{\bar\nabla} \bullet (a^\ast\Phi^{\bar\nabla,\cntnL}) \bullet \phi^{\cntnL}.
\end{equation}

\begin{prop}
There exists an open neighborhood $U$ of constant paths in $P_gM$ such that the $L_\infty[1]$-morphism $\Psi|_a:\P^{\cntnL}\M|_a \to \P^{\bar\nabla}\M|_a$ is an isomorphism for any $a \in U$.
\end{prop}
\begin{pf}
For $a \in P_gM$, the linear part $\Psi_1|_a$ is given by   
\begin{align*}
\Psi_1|_a & = \tilde\pi^{\bar\nabla}_1  \id  \phi^{\cntnL}_1  \\
& = \pi^{\bar\nabla}( \id + \nu^{\bar\nabla}_1\eta^{\bar\nabla})^{-1} (\id + \eta^{\cntnL} \nu^{\cntnL}_1)^{-1} \iota^{\cntnL} 
\qquad (\text{by Remark~\ref{rmk:0th1stMapInHptTransfer}}) \\
& = \pi^{\bar\nabla}( \id - \nu^{\bar\nabla}_1\eta^{\bar\nabla}) (\id - \eta^{\cntnL} \nu^{\cntnL}_1) \iota^{\cntnL}  \\
& = \pi^{\bar\nabla}\iota^{\cntnL} - \pi^{\bar\nabla}\nu^{\bar\nabla}_1\eta^{\bar\nabla}\iota^{\cntnL} - \pi^{\bar\nabla}\eta^{\cntnL} \nu^{\cntnL}_1 \iota^{\cntnL}  \qquad   \\
& = \pi^{\bar\nabla}\iota^{\cntnL}.    
\end{align*}
Here we used the equations $\eta^{\bar\nabla}\eta^{\cntnL} = 0$, $\pi^{\bar\nabla}\nu^{\bar\nabla}_1\eta^{\bar\nabla} =0$ and $\pi^{\bar\nabla}\eta^{\cntnL} =0$.  
The last two equations can be proved by a similar argument as in the proof of Lemma~\ref{lem:lambda}. 

Now let $a: I \to M$ be a constant path at $P\in M$. 
Under the identification in Remark~\ref{rmk:DerPathSpAsVB}, the map $\Psi|_a = \pi^{\bar\nabla}\iota^{\cntnL}$ is the identity map and hence an isomorphism. Therefore, $\Psi|_a$ is an isomorphism in an open neighborhood of constant paths. 
\end{pf}

The following example shows that $\Psi|_a$ might not be an isomorphism when $a$ is far away from constant paths.

\begin{ex}
Let $M=\rr$, $L = M \times \rr^2[-1]$. Let $\nabla$ and $\bar\nabla$ be the connections with the properties:
\begin{gather*}
\nabla_{\partial_t}e_1 = -2\pi \, e_2, \quad \nabla_{\partial_t}e_2 = 2\pi \, e_1; \\
\bar\nabla_{\partial_t}e_1 = 0, \quad \bar\nabla_{\partial_t}e_2 = 0; 
\end{gather*}
where $\{e_1, e_2\}$ is the standard frame for $L$. For the path $a(t)=t$, the $\nabla$-constant paths $\Gamma^{\nabla}_\con(I,a^\ast L)$ over $a$ is the subspace in $\Gamma(I,a^\ast L) \cong C^\infty(\rr, \rr^2)$ spanned by the basis 
\begin{align*}
e^{\nabla}_1 &= \hphantom{-} \cos(2\pi t) e_1 + \sin(2\pi t) e_2, \\
e^{\nabla}_2 &= -\sin(2\pi t) e_1 + \cos(2\pi t) e_2. 
\end{align*}
Then 
$$
\Psi_1|_a(e_1^{\nabla} \, dt) = \pi^{\bar\nabla}\iota^\nabla(e_1^\nabla \, dt) 
 = \left( \Big(\int_0^1\cos(2\pi s) \, ds \Big) \, e_1 + \Big(\int_0^1 \sin(2\pi s) \, ds \Big)\, e_2 \right) \, dt = 0.
$$
In fact, one can show that $\Psi_1|_a$ is isomorphic to the projection operator 
$$\rr \, dt  \oplus \rr^2[-1] \, dt  \oplus \rr^4[-1] \to \rr \, dt  \oplus \rr^2[-1] \, dt  \oplus \rr^4[-1], \quad v\, dt\, + y\, dt\, + x \mapsto v\, dt\, + x.
$$
\end{ex}

\subsection{Further remarks}

As we have seen in the previous subsections, a model for the derived path space $\P\M$ of the $L_\infty$-bundle $\M=(M,L,\lambda)$ is given by the graded vector bundle
$$\P\M=P_\con(TM\oplus L)\,dt\oplus P_\lin L\,,$$
over the manifold $P_gM$.  Here $P_\con(TM\oplus L)\,dt$ is a curved $L_\infty[1]$-ideal in $\P\M$.  The fibration $\P\M\to\M\times \M$ is given by dividing out by this $L_\infty[1]$-ideal, and identifying the quotient $P_\lin L$ with $\ev_0^\ast L\oplus \ev_1^\ast L$, over the embedding $P_gM\subset M\times M$. 

Note that $P_\lin L\subset \P\M$ is not a curved $L_\infty[1]$-subalgebra, as the embedding does not preserve the curvature.

If $\M$ is of amplitude $n$, then $\P\M$ is an $L_\infty$-bundle of amplitude $n+1$, and thus for degree reasons, the $k$-th brackets on $\P\M$ vanish for all $k \geq n+1$.

The curvature on $\P\M$ is given in \eqref{eq:TransferredCurvature}. 
In general, the operations on $\P\M$ can be computed by the recursive formulas \eqref{eq:TransInc} and \eqref{eq:TransLoo}, or by the tree formulas in Remark~\ref{rmk:TreeFormula}. More explicitly, the operations $\delta+\nu$ on $\P\M$ decompose into a part that has image in $P_\lin L$, and a part that has image in $P_\con(TM\oplus L)\,dt$.  The part that has image in $P_\lin L$ is the term $\pi(a^\ast\lambda) \bullet \phi$ in \eqref{eq:TransLoo} and is identified with $\ev_0^\ast L\oplus \ev_1^\ast L$, as seen in Lemma~\ref{lem:lambda}. So the only interesting part is the one which has image in $P_\con(TM\oplus L)\,dt$.  The tree sum for this part is such that the operation on every node is $\cntnL\lambda+\tilde\lambda$ with exactly one of the incoming edges is labeled with an input from $P_\con(TM\oplus L)\,dt$.

Let us work out a few examples.  The case where $\M$ is a manifold was treated in Section~\ref{sec:coam}. 

\begin{ex}
Suppose that $L=L^1$. This is the {\em quasi-smooth }case. 
Then the derived path space of $\M=(M,L,\lambda_0)$ is the following  $L_\infty$-bundle of amplitude~2: the base is $P_gM$, a manifold of short geodesic paths. The vector bundle in degree 1 has fiber $\Gamma_\con(I,a^\ast TM) \, dt   \oplus \Gamma_\lin(I,a^\ast L)$ over the geodesic path $a:I \to M$. The curvature
is $a'dt \in \Gamma_\con(I,a^\ast TM) \, dt$, and the linear interpolation between
$\lambda_0|_{a(0)}$ and $\lambda_0|_{a(1)}$ in $\Gamma_\lin(I,a^\ast L)$. In degree 2, the vector
bundle has fiber $\Gamma_\con(I, a^\ast L)\,dt$, and the twisted differential is given by 
\begin{align*}
\Gamma_\con(I,a^\ast TM) \, dt & \longrightarrow \Gamma_\con(I, a^\ast L)\,dt \\
v \, dt & \longmapsto \Big(\int_0^1 \cntnL_{v(u)}\lambda_0(u)\,du \Big)\, dt
\end{align*}
and
\begin{align*}
\Gamma_\lin(I, a^\ast L)& \longrightarrow \Gamma_\con(I, a^\ast L)\,dt\\
\alpha & \longmapsto \delta\alpha = -\big(\alpha(1)-\alpha(0)\big)\,dt \, .
\end{align*}
All the higher operations vanish.

In particular, if $L=L^1 = M \times V$ is a trivial vector bundle, the sections of $L$ are naturally identified with the space $C^\infty(M,V)$ of smooth maps from $M$ to $V$, and so a vector field on $M$ acts on the sections of $L$ by derivatives. This action defines the trivial linear connection $\nabla$ on $L$. With this trivial connection, the space $\Gamma_\con(I, a^\ast L)$ is identified with  the space $C^\infty_\con(I,V) \cong V$ of constant paths from $I$ to $V$, and the space $\Gamma_\lin(I, a^\ast L)$ is identified with the space $C^\infty_\lin(I,V)=\{ \alpha:I \to V \mid \alpha(u)=cu+d \text{ for some } c,d \in \rr\}$.  The twisted differential is then given by 
\begin{align*}
\Gamma_\con(I,a^\ast TM) \, dt & \longrightarrow V \,dt, \; v \, dt \longmapsto \Big(\int_0^1 v(u)(\lambda_0) \,du \Big)\, dt\, , \\
\Gamma_\lin(I, a^\ast L) & \longrightarrow V \,dt, \; 
\alpha  \longmapsto  -\big(\alpha(1)-\alpha(0)\big)\,dt \, .
\end{align*}
Here $v(u)$ is a tangent vector of $M$ at $a(u)$, and the curvature $\lambda_0$ is considered as a smooth map $\lambda_0:M \to V$. 
\end{ex}

\begin{ex}
If $\M=(M,L,\lambda)$ is of amplitude 2, the only non-zero operations are the curvature $\lambda_0$ and the differential $\lambda_1$.  In this case, $\P\M$ has amplitude 3.  
The curvature and the twisted differential are described in \eqref{eq:01inputs}. The binary bracket on $\P\M$ is the sum of two operations
$$\pi_\con\cntnL\lambda_1:\Gamma_\con(I,a^\ast TM) \, dt\otimes \Gamma_\lin(I, a^\ast L^1)\longrightarrow \Gamma_\con(I, a^\ast L^2)\,dt$$
and the operation 
$$\Gamma_\con(I,a^\ast TM) \, dt\otimes \Gamma_\con(I,a^\ast TM) \, dt \longrightarrow \Gamma_\con(I, a^\ast L^2)\,dt$$
given by 
$$x\otimes y\longmapsto - \pi_\con(\cntnL\lambda_1)\big(\eta\cntnL\lambda_0(x),y\big) - \pi_\con(\cntnL\lambda_1)\big(x,\eta\cntnL\lambda_0(y)\big) \,.$$
All the higher operations vanish due to degree reason.
\end{ex}

\begin{rmk}
It   would be interesting to extend the above  construction of derived path spaces to that of
``higher derived path spaces", and study their relation with
the infinite category structure of $L_\infty$-bundles.
\end{rmk}

\subsection{Homotopy fibered products and derived intersections}\label{sec:HptFibProd}

As an important application, Theorem \ref{thm:main} allows us
to form the homotopy fibered product  of $L_\infty$-bundles. Recall that given arbitrary morphisms $X \to Z \leftarrow Y$ in a category of fibrant objects $\C$, one can form a homotopy fibered product $X \times_Z^h Y$ in the homotopy category $\Ho(\C)$ \cite{BLX2,MR341469}. It is represented by the object $X \times_Z P \times_Z Y$ in $\C$, where $P$ is a path space object of $Z$. See Section~\ref{sec:CFO}.

As we see below, the usual dimension formula still holds for homotopy fiber products of $L_\infty$-bundles.

\begin{prop}
Let $\M \to \N \leftarrow \M'$ be arbitrary morphisms of $L_\infty$-bundles. Then
$$
\ddim(\M \times^h_\N \M') = \ddim(\M) + \ddim(\M') - \ddim(\N)\,.$$
\end{prop}
\begin{pf}
Let $\P$ be a path space object of $\N$. Since there is a weak equivalence $\N \xrightarrow\sim \P$, the virtual dimensions of $\N$ and $\P$ are equal. 
Thus, by Remark~\ref{rmk:DerDim&FibProd}, we have 
\begin{align*}
\ddim(\M \times^h_\N \M') & = \ddim(\M \times_\N \P \times_\N \M') \\
& = \ddim(\M) - \ddim(\N) + \ddim(\P) -\ddim(\N) + \ddim(\M') \\
& = \ddim(\M) + \ddim(\M') - \ddim(\N)\,.
\end{align*}
\end{pf}

\subsubsection*{Derived intersections}

As an application of homotopy fiber product, we consider the derived intersection of submanifolds. 

Let $M$ be a smooth manifold, and $X,Y$ be submanifolds of $M$. The \emph{derived intersection} $X \cap^h_M Y$ of $X$ and $Y$ in $M$ is understood as the homotopy fibered product $X \cap^h_M Y := X \times_M^h Y$ in the homotopy category of $L_\infty$-bundles. According to Section~\ref{sec:coam}, the quasi-smooth $L_\infty$-bundle $\P=(P_gM,P_\con TM\,dt,D)$ is a path space object of $M$. Thus, the derived intersection $X \cap^h_M Y$ is represented by the quasi-smooth $L_\infty$-bundle 
$$
X \times_M \P \times_M Y = (\baseDI,\totalDI,\strDI).
$$ 
Here, the base space $\baseDI = X \times_M P_g M \times_M Y$ is a space of short geodesic paths which start from a point in $X$ and end at a point in $Y$. Since $P_gM$ can be identified with an open neighborhood of the diagonal in $M \times M$, the base space $\baseDI$ can be regarded as an open submanifold of $X\times Y$
consisting of pairs of points $(x , y)\in X\times Y$
such that both $x$ and $y$ are sufficiently close to
the set-theoretical intersection $X \cap Y$. 
 For any $a \in \baseDI$, the fiber $\totalDI|_a$ is the space of covariant constant paths in $TM$ over the short geodesic $a$, i.e.
$$
\totalDI|_a = \{ \alpha\, dt \mid \alpha \in \Gamma(a^\ast TM), \; (a^\ast\cntnM) (\alpha) = 0 \} \,,
$$ 
and the section $\strDI: \baseDI \to \totalDI: a \mapsto a' \, dt$ is given by derivatives. 
The classical locus of $X \cap^h_M Y$ is the set-theoretical intersection $X \cap Y$. 
The virtual dimension of $X \cap^h_M Y$ is 
\begin{align*}
\ddim(X \cap^h_M Y) 
 & = \dim(X) + \dim(Y)- \dim(M) \,.
\end{align*}

\begin{lem}\label{lem:DerIntersectionTC}
The tangent complex of the $L_\infty$-bundle $(\baseDI,\totalDI,\strDI)$ at a classical locus $P \in X \cap Y$ is the two-term complex 
$$
\xymatrix{
TX|_P \oplus TY|_P \ar[r]^-{d_P} & TM \, dt|_P,
}
$$
where $d_P(v,w)=(w -v)\, dt$, and $dt$ is a formal variable of degree $1$.
\end{lem}
\begin{pf}
Choose an affine connection $\cntnM$ on $M$. Let $U$ be the open neighborhood of the diagonal in $M \times M$ which is diffeomorphic to $P_gM$ via the diagram \eqref{eq:PgM}. The base manifold $\baseDI = X \times_M P_g M \times_M Y \cong X \times_M U \times_M Y$ is naturally identified with an open neighborhood of $(X \cap Y) \times (X \cap Y)$ in $X \times Y$. With this identification, the vector bundle $\totalDI \to \baseDI$ is the pullback of the bundle $TM \, dt \to M$ via the map $\baseDI \hookrightarrow X \times Y \twoheadrightarrow X \hookrightarrow M$. The section $\strDI: \baseDI \to \totalDI$ is given by $\strDI|_{(x,y)} =  (\exp_x^{\cntnM})^{-1}(y)\, dt \in TM\, dt |_x $. 

Let $P \in X \cap Y$ be a classical locus of $X \cap^h_M Y$.  It is identified with $(P,P) \in \baseDI$ when $\baseDI$ is regarded as an open submanifold of $X \times Y$. It is simple to see that the composition 
$$
T(X\times Y)|_{(P,P)} \xrightarrow{((\exp^{\cntnM})^{-1})_\ast} T(TM)|_{0_P} \xrightarrow{\pr} T^\fiber_{0_P}(TM) \cong T M|_P
$$ 
sends $(v,w) \in TX|_P \oplus TY|_P$ to $w-v \in TM|_P$, where $(\exp^{\cntnM})^{-1}: X \times Y \to TM$ is the map sending $(x,y)$ to $(\exp^{\cntnM}_x)^{-1}(y)$. 
Hence we conclude that the tangent complex of $X\cap^h_M Y$ at $P$ is the desired two-term complex.
\end{pf}

\begin{prop}\label{prop:DerIntersection}
If  $X$ and $Y$ intersect transversally in $M$, then the map $X \cap Y \to X \cap^h_M Y$ induced by constant paths is a weak equivalence. 
\end{prop}
\begin{pf} 
Since $X$ and $Y$ intersect transversally in $M$, the intersection $X \cap Y$ is a submanifold of $M$. Let $(\baseDI,\totalDI,\strDI)$ be the model for $X \cap^h_M Y$ as in Lemma~\ref{lem:DerIntersectionTC}, where the base manifold $\baseDI$ is identified with an open submanifold of $X \times Y$. 
With this identification, the map of constant paths can be identified with the diagonal map $\Delta: X \cap Y \to \baseDI,$ $P \mapsto (P,P)$. 
Furthermore, the map $\Delta$ defines a morphism of $L_\infty$-bundles which induces a bijection on classical loci. 

The induced map on tangent complexes at $P \in X \cap Y$ is the vertical cochain map
$$
\xymatrix{
 T(X \cap Y)|_P \ar[r] \ar[d]^-{\Delta_\ast} & 0 \ar[d] \\
 TX|_P \oplus TY|_P \ar[r]^-{d_P} & TM\, dt |_P.
}
$$
Since $X$ and $Y$ intersect transversally, the map $d_P:TX|_P \oplus TY|_P \to TM|_P$ is surjective according to Lemma~\ref{lem:DerIntersectionTC}.  Thus, the vertical cochain map is a quasi-isomorphism at each classical locus $P \in X \cap Y$, as desired.
\end{pf}

Let $\pi:E \to M$ be a vector bundle (of degree zero) over $M$. One can naturally associate a section $s \in \Gamma(E)$ with two $L_\infty$-bundles: the derived intersection $M \cap^h_E s(M)$ and the quasi-smooth $L_\infty$-bundle  $(M,E\, dt, s\, dt)$. Here the symbol $M$ in $M \cap^h_E s(M)$ denotes the image of the zero section of $M$ in $E$, and the symbol $dt$ is a formal variable of degree $1$. 
It is natural to ask what the relation between these two $L_\infty$-bundles is. 
To compare them, we need the following  

\begin{lem}\label{lem:SpecialCntn}
Consider $E$ as a manifold. There exists an affine connection $\nabla^E$ on $E$ such that  
$$
\exp^{\nabla^E}_{0_P}(e) = e,
$$ 
where $e \in E|_P$ is identified with the corresponding vertical tangent vector in $TE|_{0_P}$. 
\end{lem}
\begin{pf}
Let $\{U_\lambda\}_{\lambda \in I}$ be an open cover of $M$ which trivializes the vector bundle $E$: $E|_{U_\lambda} \cong U_\lambda \times V$. Choose an affine connection $\nabla^{U_\lambda}$ on $U_\lambda$, and extend it to an affine connection $\nabla^\lambda$ on $U_\lambda \times V$ by 
$$
\nabla^\lambda_{\frac{\partial}{\partial x^i}} \frac{\partial}{\partial x^j} = \nabla^{U_\lambda}_{\frac{\partial}{\partial x^i}} \frac{\partial}{\partial x^j}, \qquad \nabla^\lambda_{\frac{\partial}{\partial x^i}} \frac{\partial}{\partial \xi^j} = \nabla^\lambda_{\frac{\partial}{\partial \xi^i}} \frac{\partial}{\partial x^j} = \nabla^\lambda_{\frac{\partial}{\partial \xi^i}} \frac{\partial}{\partial \xi^j} = 0,
$$
where $\{x^i\}$ is a local coordinate system on $U_\lambda$, and $\{\xi^j\}$ is a basis for $\dual{V}$ regarded as a coordinate system on $V$.  Let $\mu_\lambda$ be a partition of unity on $M$ subordinate to the open cover $\{U_\lambda\}_{\lambda \in I}$. By pulling back it to $E$, we have the partition of unity $\varphi_\lambda := \mu_\lambda \circ \pi$ on $E$. Define an affine connection $\nabla^E$ on $E$ by 
$$
\nabla^E = \sum \varphi_\lambda \nabla^{\lambda}.
$$
It is easy to show that 
$\nabla^E_X Y =0$ for any $X,Y \in\Gamma(E)$ regarded as vertical vector fields on $E$. Thus, the paths $a(t) = 0_P + t e$, $e \in E|_P$, are geodesics, and $\exp^{\nabla^E}_{0_P}(e) = e$.
\end{pf}

\begin{prop}
Let $E$ be a vector bundle over $M$, and $s$ a section of $E$. The derived intersection $M \cap^h_E s(M)$ is weakly equivalent to the quasi-smooth $L_\infty$-bundle $(M,E\, dt, s\, dt)$.  
\end{prop}
\begin{pf}
Let $(\baseDI,\totalDI,\strDI)$ be a representation of $M \cap^h_{ E} s(M)$ associated with the connection $\nabla^E$ as in Lemma~\ref{lem:SpecialCntn}. 
Consider the morphism of $L_\infty$-bundles
\begin{equation}\label{eq:WkEqBwTwoQSmooth}
(f,\phi):(M,E\, dt, s\, dt) \to (\baseDI,\totalDI,\strDI),
\end{equation}
where $f$ sends $P \in M$ to the short geodesic $a(t) = 0_P + t \cdot s(P)$ in $E$, and the bundle map $\phi:E \, dt \to \totalDI$ sends $e\, dt \in E\, dt|_P$ to the covariant constant path $\alpha$ over $a$ satisfying $\alpha(0) = e\, dt$.  Here we identify $e \in E|_P$ with the corresponding vertical tangent vector in $TE |_{0_P}  \cong TM|_P \oplus  E|_P$.

By Lemma~\ref{lem:DerIntersectionTC}, the tangent complex of $(\baseDI,\totalDI,\strDI)$ at a classical locus $P \in Z( s)$, is the two-term complex 
$$
\xymatrix{
TM|_P \oplus TM|_P \cong TM|_P \oplus T \big(s(M)\big)|_P  \ar[r]^-{d_P} & TE\, dt |_{0_P} \cong  TM \, dt |_P \oplus  E\, dt|_P,
}
$$ 
where $d_P(v,w) = \big((w-v)\, dt, (D_Ps)(w)\, dt \big)$ for $v,w \in TM|_P$. 
Furthermore, the induced morphism of the $L_\infty$-bundle morphism $(f,\phi)$ in \eqref{eq:WkEqBwTwoQSmooth} on the tangent complexes is the vertical cochain map 
$$
\xymatrix{
TM|_P \ar[r]^-{D_P s\, dt} \ar[d]_-{f_\ast} & E\, dt|_P \ar[d]^-{\phi} \\
 TM|_P \oplus TM|_P \ar[r]_-{d_P} & TM \, dt |_P \oplus  E\, dt|_P,
}
$$ 
where  $f_\ast(v)= (v,v)$ and $\phi(e \, dt) = (0 , e \, dt)$. 
With these formulas, it is clear that the morphism $(f,\phi):(M,E\, dt,s\, dt) \to (\baseDI,\totalDI,\strDI)$ is a weak equivalence. 
\end{pf}

See \cite{2023arXiv231216622S} for further discussions on the intersection of $M$ and $s(M)$.

\begin{rmk}
We define an {\it orientation} of a quasi-smooth $L_\infty$-bundle $(M,E,s)$ with classical locus $X$ to consist of orientations of both $M$ and $E$, i.e. Thom classes for $TM$ and $E$.  As explained in \cite{MR609831}, such Thom classes give rise to bivariant classes for $0:M \to E$ and for $M \to \ast$.  Pulling back the bivariant class for $0:M \to E$ via the section $s$ to a bivariant class for $X \to M$, and composing with the class for $M \to \ast$, we get a bivariant class for $X \to \ast$, in other words a Borel-Moore homology class of $X$.  This Borel-Moore homology class is the {\it virtual fundamental class} of $(M,E,s)$.  It is a class of degree $\ddim (M,E,s) = \dim M - \rk E$. It is also invariant under oriented weak equivalences (i.e. weak equivalences compatible with the orientations).  

If $X$ is compact, and the virtual dimension of $(M,E,s)$ is zero, we can push the virtual fundamental class forward to $\ast$, and obtain an integer, the {\it virtual number of points} of $(M,E,s)$. 

We can generalize the above considerations to the relative situation, where $S$ is a manifold, and the quasi-smooth $L_\infty$-bundle $(M,E,s)$ is endowed with a submersion $M \to S$.  An orientation is now a pair of a Thom class for $E$, and one for the relative tangent bundle $TM/S$.  We obtain a relative virtual fundamental class, which is a bivariant class for the morphism $X \to S$.  It will pull back to the virtual fundamental classes of the various fibres.  If $X \to S$ is proper, and $S$ is connected, it will follow that the virtual number of points of all fibres are equal. 

In particular, we see that derived intersections, if their classical loci are compact, have a well-defined intersection number, the virtual number of points of the associated quasi-smooth $L_\infty$-bundle.  This intersection number will be constant in smooth connected families, as long as the classical intersection remains compact in the family.
\end{rmk}

\appendix

\section{Path space construction}
\label{sec:AKSZ}
In this section, we present constructions
of infinite dimensional dg manifolds and $L_\infty$-bundles
 of path spaces. 
The discussion here aims to give a heuristic
reasoning as to where the curved $L_\infty[1]$-structure formula
in Proposition \ref{pro:paris} comes from.

We will not address the subtle issue such as in which
sense a path space is  an  infinite dimensional smooth
manifold. It can be understood as a smooth manifold in the sense of \cite{MR583436}.

\subsection{Twisted shifted tangent $L_\infty$-bundles}

Let $\M = (M,L,\lambda)$ be an $L_\infty$ bundle. We start with a curved $L_\infty$ structure on $TL[-1]$ which is needed later.   
Recall that a  degree $0$  vector field $\tau\in  \XX(L)$  is said to be
{\it  linear} \cite{MR1617335} if  $\tau$ is projectible, i.e.
$\pi_* (\tau)=X\in \XX(M)$ is a well-defined
vector field on $M$, and, furthermore, the diagram
$$
\xymatrix{
L \ar[d] \ar[r]^{\tau} & TL \ar[d] \\
M \ar[r]_{X} & TM
}
$$
is a morphism of vector bundles. Geometrically, linear vector fields
correspond exactly  to those whose flows are automorphisms
 of the vector bundle $L$.

Choose a linear connection $\cntnL$ on $L$. Given a linear vector field  $\tau\in  \XX(L)$, 
consider the map $ \tilde{\delta}  : L\to L[-1]$
defined by
$$\tilde{\delta} (l)=\tau|_l -\widehat{X(\pi (l))}|_l, \  \ \ \forall l\in L,$$
where $\widehat{X(\pi (l))}|_l\in TL|_l $ denotes the horizontal lift of $X(\pi (l))\in TM|_{\pi (l)}$ at $l$.

\begin{lem}
\label{lem:Rome}
The map $\tilde{\delta}  : L\to L[-1]$ is a bundle map.
\end{lem}
\begin{pf}
 Let $\alpha$ be a path in $M$ such that $\alpha(0) = \pi(l), \alpha'(0) 
= X(\pi(l))$, and let $\gamma(l,t)$ be the parallel transport of $l$
 along $\alpha(t)$. Then the horizontal lift of $X$ at $l$ is
 $\gamma'(l,0)$. Since the parallel transport
 $L|_{\alpha(0)} \to L|_{\alpha(t)}: l \mapsto \gamma(l,t)$
 is linear at each $t$, the horizontal lift of $X$ at $l+\tilde l$ is
 $\gamma'(l + \tilde l,0) = \gamma'(l,0) + 
\gamma'(\tilde l,0)$.
Hence $\tilde \delta(l+\tilde l) = \tau(l+\tilde l) - ( \gamma'(l,0) + \gamma'(\tilde l,0)) = \tilde\delta(l)+ \tilde\delta(\tilde l)$.
\end{pf}

 Consider the dg manifold 
$$(TL[-1], \ \  \iota_\tau)$$
Again, a priori, this dg manifold is not an $L_\infty$-bundle over $M$.
However, if $\tau$ is a linear vector field on $L$,
 a linear connection $\cntnL$ on
$L$ will make it into an $L_\infty$-bundle.

The following proposition follows from a straightforward
verification, which is left to the reader.

\begin{prop}
\label{pro:TM2}
Let $L$ be a positively graded vector bundle, and $\tau \in \calx (L)$
a linear   vector field on $L$. Then any linear connection $\cntnL$ on $L$ 
induces an $L_\infty$-bundle structure on $(M, TM\, dt\oplus L\, dt\oplus L, \nu)$,
where
$$\nu=X+{\delta}. $$
Here $X=\pi_* (\tau) \in \Gamma (M, TM \, dt\, )$ is the curvature, and 
${\delta}$ is the degree 1 bundle map on $TM\, dt\oplus L\, dt\oplus L$
naturally extending the bundle map $ \tilde{\delta}: L\to L\, dt$
as in Lemma \ref{lem:Rome}.

Moreover,  different choices of linear connections $\cntnL$ on $L$ induce
isomorphic $L_\infty$-bundle  structures on $TM\, dt\oplus L\, dt\oplus L$.
\end{prop}

Below is the situation
which is a combination of Proposition \ref{pro:TM1} and Proposition \ref{pro:TM2}. 
\begin{prop}
\label{pro:TM3}
Let $\M=(M,L,\lambda)$ be an $L_\infty$-bundle with $Q$ 
being its corresponding homological vector field.
Let $\tau \in \calx (L)$ be a degree $0$ linear  vector field on
 $L$ such that
$\pi_* (\tau)=X\in \calx (M)$.
Assume that 
\begin{equation}
\label{eq:tauQ}
[\tau , \ Q]=0.
\end{equation}
Choose a linear connection $\cntnL$ on $L$. Then there is an induced
  $L_\infty$-bundle structure on $(M, TM\, dt\oplus L\, dt\oplus L, \nu+\mu)$,
 where $\mu$ and $\nu$ are described as in Proposition \ref{pro:TM1} and Proposition \ref{pro:TM2}, respectively.

Moreover,  different choices of linear connections $\cntnL$ on $L$ induce
isomorphic $L_\infty$-bundle  structures on $TM\, dt\oplus L\, dt\oplus L$.
\end{prop}
\begin{pf}
Consider the graded manifold $TL[-1]$. Both $\iota_\tau$  and $\LL_Q$
are homological vector fields on $TL[-1]$. Now
\begin{eqnarray*}
[\iota_\tau+\LL_Q, \ \iota_\tau+\LL_Q]
&=&2[\iota_\tau, \LL_Q]\\
&=&\iota_{[\tau, Q]}\\
&=&0.
\end{eqnarray*}
Therefore $\big( TL[-1], \iota_\tau+ \LL_Q \big)$ is indeed
 a dg manifold.  Choosing a linear connection $\cntnL$ on $L$,  one can identify $TL[-1]$
with $TM\, dt\oplus L\, dt\oplus L$ via the map $\phi^{\cntnL}$ as
in \eqref{eq:phinabla}.
The rest of the proof follows exactly from that of
Proposition \ref{pro:TM1} and Proposition \ref{pro:TM2}.
\end{pf} 

\begin{rmk}
\label{rmk:tauQ}
Note that $\tau$ is a degree $0$ vector field on $L$, while
$Q$ is of degree $1$. The commutator $[\tau , \ Q]$ is a
vector field of degree $1$. By $\phi_t$, we denote the (local)
flow on $L$   generated by $\tau$. Then \eqref{eq:tauQ}
is equivalent to saying that $\phi_t$ is a family of (local)
isomorphisms  of the dg manifold $(\M, Q)$. 
Since  $\tau$ is a linear vector field,
$\phi_t$ is a family automorphisms of the vector bundle $L$.
Then \eqref{eq:tauQ} is equivalent to saying that 
the  (local) flow $\phi_t$ is a family of isomorphisms  of the $L_\infty$-bundle
$\M=(M,L,\lambda)$.
\end{rmk}

\subsection{Mapping spaces: algebraic approach}

Let $L = \bigoplus_{i=1}^n L^i$ be a graded vector bundle over a manifold $M$, and $\M = (M, \A)$ the associated graded manifold of amplitude $[1,n]$. For any $P \in M$, $l \in L|_P$, one has the evaluation map 
$$
\ev_l: \Gamma(M,\A) \xrightarrow{\ev_P} \Sym L^\vee|_P \to \rr,
$$
where the second map is the evaluation of a polynomial on $L|_P$ at $l$.

\begin{lem}
The evaluation map defines a bijection from $L$ to the algebra morphisms (which may not preserve degrees) from $\Gamma(M,\A)$ to $\rr$.  
\end{lem}
\begin{pf}
It is clear that $\ev_l$ is an algebra morphism. Conversely, let $\phi: \Gamma(M,\A) \to \rr$ be an algebra morphism. Then $\phi$ is uniquely determined by its restriction to the generators $\Gamma(M,\O_M) \oplus \Gamma(M,L^\vee)$. Decomposing the restriction according to degrees, we have
\begin{items}
\item
an algebra morphism $\phi_0: \Gamma(M,\O_M) \to \rr$, and
\item
$\phi_i: \Gamma(M,{L^i}^\vee) \to \rr$ such that 
$$\phi_i(ab) = \phi_0(a) \phi_i(b),$$ 
for any $a \in \Gamma(M,\O_M)$, $b \in \Gamma(M,{L^i}^\vee)$.  
\end{items}
It is well-known that such an algebra morphism $\phi_0$ is of the form $\ev_P$ for some $P\in M$, and such a map $\phi_i$ is equal to $\ev_l$ for some $l \in L^i|_P$.
\end{pf}
 
The shifted tangent bundle $TL[-1]$ can be understood as a mapping space in the following way.

\begin{prop}\label{prop:PathTM}
There is a bijection between $TL[-1]$ and the set of algebra morphisms from $\Gamma(M,\A)$ to $\Gamma(\ast, \rr[1]) \cong \rr \oplus \rr[-1]$. 
\end{prop}
\begin{pf}
Let $\psi:\Gamma(M,\A) \to \rr \oplus \rr[-1]$ be any algebra morphism, and let $\psi^0$ and $\psi^1$ be the compositions of $\psi$ followed by the projections onto $\rr$ and $\rr[-1]$, respectively. Since $\psi$ is algebra morphism, so is $\psi^0$. Thus $\psi^0 = \ev_l$ for some $l\in L$. 

Since $\psi$ is an algebra morphism, from the algebra morphism property
$$
(\psi^0 +\psi^1)(ab) = (\psi^0 +\psi^1)(a)(\psi^0 +\psi^1)(b),
$$
it follows that 
\begin{align*}
\psi^1(a b)&  =  \psi^1(a) \psi^0(b) + \psi^0(a) \psi^1(b) \\
&=  \psi^1(a) \ev_l(b) + \ev_l(a) \psi^1(b)
\end{align*}
for any $a, \, b \in \Gamma(M,\A)$. 
Therefore $\psi^1$ can be identified with a tangent vector in $TL|_l$. This completes the proof. 
\end{pf}

Identifying mapping spaces with algebra morphisms of function algebras in the converse direction, we therefore have that
\begin{equation}
\label{eq:Map(R[1],M)}
\Map(\rr[1],\M) \cong T\M[-1].
\end{equation}

Note that \eqref{eq:Map(R[1],M)} is well-known. Here we  follow
the algebraic approach as in
\cite[Section 3.2.2]{helein2020introduction}.

\subsection{Motivation: AKSZ} 
Recall the following folklore concerning graded manifolds
 and dg manifolds following AKSZ \cite{MR1432574,MR2819233,2003math......7303K}. Here we follow \cite{MR2819233}
closely.  Let $\N$ and $\M$ be graded manifolds. Then
 $\Map (\N, \M)$ is a (usually infinite dimensional) graded manifold satisfying
\begin{align*}
\Mor(\Z\times \N, \M) \cong \Mor(\Z,\Map(\N, \M)).
\end{align*}

If, moreover, both $\N$ and  $\M$ are dg manifolds, then  homological
vector fields on $\N$ and $\M$ induce a homological vector field $Q_\mapnm$
so that $\big(\Map(\N, \M), \  Q_\mapnm\big)$ is a dg manifold.

Indeed, we have
\begin{align*}
Q_\mapnm:=\phi^L ( Q_\N ) + \phi^R (Q_\M),
\end{align*}
where $Q_\N$ and  $Q_\M$ denote the  homological vector fields
on   $\N$ and $\M$, respectively, and $\phi^L$ and  $\phi^R$
are the natural morphisms:
\begin{align*}
\xymatrix{\calx(\N)  \ar[rrr]^-{\phi^L} &&& \calx(\Map(\N, \M)) &&& \calx(\M) \ar[lll]_-{\phi^R}.}\end{align*}

Now let $I=(a,b)$ be a fixed open interval containing $[0,1]$ as in Section~\ref{sec:ShortGeodesic}. Let $\N=TI[1]$ be the natural dg manifold with
the homological vector field $Q_\N$ being the
de Rham differential, and $\M=(M,L,\lambda)$ be an $L_\infty$-bundle
 with homological vector field $Q$ on $\M$. Then it is known
\cite{MR2819233, 2003math......7303K} that, as a graded manifold,
\begin{eqnarray*}
\mapnm&=&\Map (TI[1] , \ \M)\\
&=&\Map (I\times \rr [1], \ \M)\\
&=& \Map \big(I, \Map ( \rr [1], \ \M) \big)\\
&\cong& \Map \big(I, T\M[-1]\big)\\
&=&PT\M[-1]\\
& \cong& T(P\M)[-1],
\end{eqnarray*}
where $P(\argument)=\Map(I,\argument)$ denotes the path space. See also \eqref{eq:Map(R[1],M)} and Proposition~\ref{prop:PathTM}.

Moreover, $\phi^L ( Q_\N )=\iota_\tau$, where $\tau \in \calx (P\M)$
is the tautological vector field on  $P\M$,
and $\phi^R (Q)=\widehat{PQ}[-1]$, the shifted complete lift  on $T(P\M)[-1]$ of
the homological vector field $PQ\in \calx (P\M)$.
As  a consequence, it follows that
 $\big(T(P\M)[-1], \ \iota_\tau+\widehat{PQ}[-1]\big)$
is an (infinite dimensional)  dg manifold. 

We summarize the
discussion above in the following

\begin{prop}
Let  $\tau \in \calx (P\M)$ be the tautological vector field on  $P\M$,
then $\big(T(P\M)[-1], \ \iota_\tau+\widehat{PQ}[-1]\big)$
is a  dg manifold.
\end{prop}

\subsection{Path space of $L_\infty$-bundles}

Indeed,  the fact above can be derived from a direct argument.  The flow
of $\tau$ is simply just a linear re-parameterization of the path space,
which preserves the  path  $L_\infty$-bundle $P\M=(PM, PL, P\lambda)$.
It thus follows that 
\begin{equation}
[\tau, \ PQ]=0,
\end{equation}
where $PQ$ denotes the corresponding homological vector field on
$P\M$. See Remark  \ref{rmk:tauQ}.

 From the proof of Proposition
\ref{pro:TM3}, it follows that $\iota_\tau+\widehat{PQ}[-1]$
is indeed a homological vector field on $T(P\M)[-1]$.

It is simple to see that $\tau \in \calx (P\M)$ is indeed a linear
vector field on the underlying vector bundle $\pi: PL\to PM$,  and
$$\pi_* \tau =D \in \calx (PM),$$
where $D$ denotes the  tautological vector field on  $PM$.
Choose a linear connection $\nabla$ on $L\to M$, which induces a linear connection on
$PL\to PM$. Now we can apply Proposition \ref{pro:TM3} formally
to the path  $L_\infty$-bundle $P\M=(PM, PL, P\lambda)$ to obtain
the path space of shifted tangent bundle.

We need a lemma first.  

Given a path $l: I\to L$  in $L$, let $a(t)=\pi (l(t))$, $t\in I$,
be its corresponding path in $M$. Then $\dot{l}(t)\in T_{l(t)}L$
and 
$$\pi_* (\dot{l}(t))=\dot{a} (t)\in T_{a(t)}M.
$$

By $\widehat{\dot{a} (t)}|_{\dot{l}(t)}\in T_{l(t)}L$, we denote
the horizontal lift of $\dot{a} (t)\in T_{a(t)}M$ at 
$l(t)$. 

The following lemma is standard. See~\cite[Theorem~12.32]{MR2572292} or~\cite[page~114]{MR0152974}.

\begin{lem}
$$\dot{l}(t) =\widehat{\dot{a} (t)}|_{\dot{l}(t)}+\nabla_{\dot{a} (t)}l(t)$$
as tangent vectors at  $T_{l(t)}L$
\end{lem}

Applying Proposition \ref{pro:TM3} formally to this situation,
 we obtain the following:

\begin{prop}
\label{pro:TM4}
Let $\M=(M,L,\lambda)$ be an $L_\infty$-bundle, and 
 $\nabla$  a  linear connection on $L$. Then
 $$(PM, PTM \, dt\, \oplus PL\, dt\,\oplus PL, D+\delta+ P\mu)$$ 
 is a bundle of curved $L_\infty[1]$-algebras (of infinite dimension), where $\mu=\lambda+\tilde\lambda+\nabla\lambda\,$ is as in
Proposition \ref{pro:TM1}, and $\delta: PL\to PL \, dt $ is the degree $1$ map
defined by  $\delta (l(t))=(-1)^{|l|} \nabla_{\dot{a} (t)}l(t) \, dt$.
\end{prop}

As a consequence,  for every path $a:I\to M$ in $M$, we obtain an induced curved $L_\infty [1]$-structure in the vector space $\Gamma(I,a^\ast(TM\oplus L)\,dt\oplus a^\ast L)$ by restricting to the fiber $a\in PM$.

As $\delta^2=0$, it induces the structure of a complex on $\pathast$.
Also note  that $D|_a$ is the derivative $a'\, dt \in \Gamma(I,a^\ast TM)\,dt$. 
 It is simple to see that the resulting
 $L_\infty [1]$-operations on $\Gamma(I,a^\ast(TM\oplus L)\,dt\oplus a^\ast L)$
  are exactly those in Proposition \ref{pro:paris}.

 We  summarize it in the following

\begin{cor}
\label{cor:paris}
Let $\delta$ be the covariant derivative of the pullback linear connection $a^\ast \nabla$ over $a \in PM$.
The sum $a'\, dt \, +a^\ast\mu$ is a curved $L_\infty[1]$-structure
on the complex $\big(\Gamma(I,a^\ast T\M[-1]), \delta \big) = \big( \pathast ,    \delta\big)$.
\end{cor}

\bibliographystyle{plain}

\bibliography{ref_CatFibObj}

\Addresses

\end{document}